\documentclass[11pt,a4paper]{amsart}
\usepackage[a4paper,centering,left=3.0cm,right=3.0cm]{geometry}
\usepackage{amsfonts,amscd,amssymb,amsmath,amsthm,mathrsfs,multirow,xcolor,lscape,longtable,pbox,lipsum}
\usepackage[thinlines]{easytable}
\usepackage{microtype}
\usepackage[normalem]{ulem}
\usepackage{etex}
\usepackage{tikz-cd,xr,multicol,microtype}
\usepackage{color}
\usepackage[colorlinks,linkcolor=blue,anchorcolor=blue,citecolor=blue,backref=page]{hyperref}
\usepackage{hyperref}
\usepackage{phaistos}
\usepackage{todonotes}

\newtheorem{thm}{Theorem}

\newtheorem{lema}[thm]{Lemma}
\newtheorem{prop}[thm]{Proposition}

\newtheorem{rmk}[thm]{Remark}

 \usepackage[all,cmtip]{xy}\usepackage{tikz}\usetikzlibrary{arrows}
\usetikzlibrary{matrix}\usetikzlibrary{calc}

\newcommand{\Z}{\mathbb{Z}}

\newcommand{\A}{\mathbb{A}}
\newcommand{\Q}{\mathbb{Q}}
\newcommand{\n}{\mathfrak{n}}
\newcommand{\G}{\Gamma}

\newcommand{\R}{\mathbb{R}}

\newcommand{\m}{\mathcal{M}}

\newcommand{\W}{\mathcal{W}}
\newcommand{\cN}{\mathcal{N}}

\newcommand{\rG}{\mathrm{G}}
\newcommand{\rGL}{\mathrm{GL}}
\newcommand{\rK}{\mathrm{K}}

\newcommand{\Ind}{\mathrm{Ind}}
\newcommand{\rP}{\mathrm{P}}
\newcommand{\rQ}{\mathrm{Q}}
\newcommand{\rN}{\mathrm{N}}

\newcommand{\rM}{\mathrm{M}}
\newcommand{\rS}{\mathrm{S}}
\newcommand{\rB}{\mathrm{B}}
\newcommand{\rT}{\mathrm{T}}
\newcommand{\rA}{\mathrm{A}}

\newcommand{\fm}{\mathfrak{m}}
\newcommand{\fg}{\mathfrak{g}}
\newcommand{\fn}{\mathfrak{n}}
\newcommand{\fp}{\mathfrak{p}}
\newcommand{\fu}{\mathfrak{u}}

\newcommand{\fa}{\mathfrak{a}}
\newcommand{\fb}{\mathfrak{b}}
\newcommand{\ft}{\mathfrak{t}}

\newcommand{\cZ}{\mathcal{Y}}
\newcommand{\cY}{\mathcal{Z}}
\newcommand{\cP}{\mathcal{P}}

\newcommand{\Eis}{\mathscr{Eis}}
\newcommand{\SL}{\mathrm{SL}}
\newcommand{\Res}{\mathrm{Res}}

\newcommand{\rank}{\mathrm{rank}}
\newcommand{\vcd}{\mathrm{vcd}}
\newcommand{\cd}{\mathrm{cd}}

\newcommand{\eqr}[1]{\mbox{(\ref{eq:#1})}}
\newcommand{\ie}{i.e.\ }
\newcommand{\mr}[1]{\mathrm{#1}}

\newcommand{\C}{\mathbb{C}}

\newcommand{\tm}{\widetilde{\m}}

\title[Cohomology of $G_{2}(\Z)$] {Boundary and Eisenstein  Cohomology of $G_2(\Z)$ }
\author{Jitendra Bajpai and Lifan Guan}
\address{Mathematisches Institut, Georg-August Universit\"at G\"ottingen, D-37073 Germany.}
\email{jitendra@math.uni-goettingen.de}
\address{Mathematisches Institut, Georg-August Universit\"at G\"ottingen, D-37073 Germany.}\email{guanlifan@gmail.com}

\begin{document}
\date{\today}

\begin{abstract}
In this article, Eisenstein cohomology of the arithmetic group $G_2(\Z)$ with coefficients in any finite dimensional highest weight irreducible representation has been determined. We accomplish this by studying the cohomology  of the boundary of the Borel-Serre compactification.
\end{abstract}

\maketitle

\tableofcontents


\section{Introduction}\label{intro}

Let $\rG$ be a semisimple algebraic group defined over $\Q$, $K_\infty \subset \rG(\R)$ be a maximal compact subgroup and $\rS = \rG(\R)/K_\infty$ be the corresponding symmetric space. If $\Gamma \subset \rG(\Q)$ is an arithmetic subgroup then every representation $(\rho, \m)$ of $\rG$ defines, in a natural way, a sheaf $\tm$ on the locally symmetric space $\rS_\Gamma = \Gamma \backslash \rS$. One has the isomorphism \begin{equation}\label{eq:gpls}
    H^\bullet(\Gamma, \m) \cong H^\bullet(\rS_\Gamma, \tm),
\end{equation} for details see Chapter 7 of~\cite{BoWa}. Note that, throughout the paper, we use $H^\bullet$ to represent the full cohomology group, namely, $H^\bullet(\G, \m)=\oplus_{q} H^q(\G, \m)$. On the other hand, let $\overline{\rS}_\Gamma$ denote the Borel-Serre compactification of $\rS_\Gamma$, then the inclusion $i:\rS_\Gamma \hookrightarrow\overline{\rS}_\Gamma$, which is an homotopic equivalence, determines a sheaf $i_\ast(\tm)$  on $\overline{\rS}_\Gamma$ and induce a canonical isomorphism in the cohomology
\begin{equation}\label{eq:hmtpyequ}
	H^\bullet(\overline{\rS}_\Gamma, i_{\ast}\tm) \cong H^\bullet(\rS_\Gamma, \tm),
\end{equation}
where $i_\ast$ denotes the direct image functor defined by $i$. On the other hand, let $\partial \rS_\Gamma= \overline{\rS}_\Gamma \setminus \rS_\Gamma$ and $j: \partial \rS_\Gamma \rightarrow \overline{\rS}_\Gamma$  be the closed embedding. The following exact sequence of sheaves
\[0\rightarrow i_!(\tm)\rightarrow i_\ast(\tm) \rightarrow j_\ast (\tm) \rightarrow 0\]
gives rise to a long exact sequence of cohomology  groups associated to $\rS_\G$,
\[\cdots \rightarrow H^q_c(\rS_\Gamma, \tm)\rightarrow H^q(\rS_\Gamma, \tm) \xrightarrow{r} H^q(\partial\rS_\Gamma, \tm)\rightarrow \cdots.\]
Let \emph{boundary cohomology of $\Gamma$ with coefficients in $\m$} denoted  by $H^\bullet(\partial\rS_\Gamma, \tm)$ and the \emph{Eisenstein of cohomology $\Gamma$ with coefficients in $\m$} denoted by $H_{Eis}^\bullet(\rS_\Gamma, \tm)$, to be the image of the map $r$. Clearly, there is an exact sequence
\begin{equation*}
 0 \rightarrow H_{!}^q(\rS_\Gamma, \tm) \rightarrow  H^q(\rS_\Gamma, \tm) \xrightarrow{r} H^q_{Eis}(\rS_\Gamma, \tm)\rightarrow 0,
 \end{equation*}
where $H_!^q(\rS_\G, \tm)$ is the kernel of the restriction map $r$.

The study of Eisenstein cohomology was initiated by Harder \cite{Harder87}, and found that  Eisenstein cohomology is fundamentally related to several important topics in number theory, e.g., special values of $L$-functions, extension of motives, to simply mention a few. See Harder's ICM report \cite{Harder91} for more details on the relation of Eisenstein cohomology with other topics. Interested reader is also referred to~\cite{HR2020} for recent advances on the subject. Although lots of work have been done, our understanding of Eisenstein cohomology is still far from complete.

The main purpose of this article is to determine the boundary and Eisenstein cohomology of the arithmetic group $G_2(\Z)$ with coefficients in any  finite dimensional highest weight representation $\m_\lambda$ of $G_2$ where $\lambda$ denotes its highest weight.

The Eisenstein cohomology of arithmetic subgroups of $\Res_{k/\Q}G_2$ for totally real field $k$ has  been previously studied in~\cite{LiSch93}. Compared to their work, the basic setting of this paper is more restrictive, that is we only consider cohomology of the full arithmetic group $G_2(\Z)$, but we provide complete results for the boundary and Eisenstein cohomology of $G_2(\Z)$ with coefficients in any finite dimensional highest weight representation. 

On the other hand, it is worth mentioning the paper~\cite{BHHM2018}, where the boundary and Eisenstein cohomology of $\SL_3(\Z)$ with coefficients in finite dimensional highest weight representations is determined by using Euler characteristic, which is a fundamental case to study among other rank two arithmetic groups. Unfortunately, their method, being elementary but tricky and powerful, does not work here in the case of $G_2(\Z)$ for dimensional reasons.

The method we are employing here follows closely the work of Harder in~\cite{Harder2012}. In particular, our method is constructive that involves the theory of Eisenstein series, intertwining operators, $(\fg, K_\infty)$-cohomology and $L$-functions. Indeed, it is well-known that the Eisenstein cohomology $H^\bullet_{Eis}(\rS_\Gamma, \tm)$ spans a maximal isotropic subspace of the boundary cohomology $H^\bullet(\partial\rS_\Gamma, \tm)$ under the Poincar\'e duality. Hence we are done if we manage to construct enough classes in $H^\bullet_{Eis}(\rS_\Gamma, \tm)$.  Starting from cohomology classes from $H^\bullet(\partial\rS_\Gamma, \tm)$, we construct cohomology classes in $H^\bullet(\rS_\Gamma, \tm)$ by evaluating the corresponding Eisenstein series at certain special point, and then we get non-trivial Eisenstein cohomology classes by restriction back to the boundary. Subtlety appears when the corresponding Eisenstein series is not holomorphic at the special point, or equivalently, the corresponding $L$-function is not holomorphic at the special point. We complete the proof by further exploring the Hecke action and Poincar\'e duality, as well as a detailed study of the corresponding $(\fg, K_\infty)$-cohomology.

\subsection{Main results} Let us now give the details of the results obtained in this article.

\begin{itemize}
\item Theorem~\ref{bdg2}, where the boundary cohomology with coefficients in every finite dimensional highest weight representation is described.

\item Theorem~\ref{Eiscoh}, where we describe the Eisenstein cohomology for every finite dimensional highest weight representation.
\end{itemize}
\subsection{Overview of the article} We quickly  summarize the content of each section of the article. In Section~\ref{preli}, we provide the details of structure of the group $G_2$: its Weyl group, parabolic subgroups, root system and Kostant representatives, which is the first basic step to accomplish our goal. In Section~\ref{parity}, the parity conditions for the cohomology of the boundary components, \ie the cohomology of the parabolic subgroups have been established which helps us to compute the boundary cohomology, content of Section~\ref{boundary},  of the locally symmetric space of the arithmetic group $G_2(\Z)$. In Section~\ref{Eis}, we achieve our main goal by completing the study of Eisenstein cohomology of $G_2(\Z)$ for every finite dimensional highest weight representation.


\section{Preliminaries}\label{preli}
This section quickly review the basic properties of $G_2$ and familiarize the reader with the notations to be used throughout the article. We discuss the corresponding locally symmetric space, Weyl group, the associated spectral sequence and Kostant representatives of the standard parabolic subgroups.

\subsection{Structure Theory}\label{g2}
Let $G_2$ be the Chevalley group defined over $\Z$ of type $G_2$  and $\Phi$ be the corresponding root system. Let us fix a maximal $\Q$-split torus $\rT$ and a Borel subgroup $\rB$ that contains $\rT$. The set of simple roots associated to $\rB$ is denoted by $\Delta=\{\alpha_1, \alpha_2\}$ with $\alpha_1$ and $\alpha_2$ be the short and long simple roots respectively. The Weyl group $\W$ of $\Phi$ is isomorphic to the dihedral group $D_6$. The fundamental weights associated to this root system are given by $\gamma_1=2\alpha_1+ \alpha_2$ and $\gamma_2=3\alpha_1+2\alpha_2$.

Let $\mathfrak{g}$ denote the Lie algebra $\mathfrak{g}_{2}$ and $\mathfrak{t} \subset \mathfrak{g}$ be the  Lie subalgebra associated to $\rT$. Let $\Phi=\Phi^+ \cup \Phi^-$ be the corresponding root system. We know that $$\Phi^+=\{\alpha_1, \alpha_2, \alpha_1+\alpha_2, 2\alpha_1+\alpha_2, 3\alpha_1+\alpha_2, 3\alpha_1+ 2\alpha_2\}.$$
  Finally, we write $\rho=\frac{1}{2}\sum_{\alpha\in \Phi^+}\alpha=5\alpha_1+3\alpha_2$.

Recall that a $\Q$-parabolic subgroup is called \emph{standard} if it contains the  Borel subgroup $\rB$.
Let $\{\rP_0, \rP_1,\rP_2\}$ be the set of standard $\Q$-parabolic subgroups, where $\rP_1$ (resp. $\rP_2$) denotes the maximal $\Q$-parabolic subgroup corresponding to the simple roots $\alpha_2$ (resp. $\alpha_1$) and $\rP_0=\rB$.  Thus,
\begin{equation*}
\fp_1=\fu_{-\alpha_1}\oplus \ft \oplus_{\alpha\in \Phi^+}\fu_{\alpha}, \text{ and } \fp_2=\fu_{-\alpha_2}\oplus \ft \oplus_{\alpha\in \Phi^+}\fu_{\alpha},
\end{equation*}
where $\fp_i$ denotes the Lie algebra of $\rP_i$ and $\fu_{\alpha}$ denotes the root space corresponding to $\alpha$. Note that the minimal  $\Q$-parabolic $\rP_0$ is simply the group $\rP_1 \cap \rP_2$. Therefore, the corresponding Levi quotients are given by
\[
\rM_0 = \mathbb{G}_m^2, \quad \rM_1 = \rGL_2 \quad \mbox{ and } \quad \rM_2 = \rGL_2.
\]

Let us choose and fix a maximal compact subgroup $K_\infty \subset G_2(\R)$. It is well-known that it  can be identified with $SO_4(\R)$.
From now on throughout the article let $\rS = G_2(\R) / K_\infty$, $\Gamma$ be the arithmetic group $G_2(\Z)$ and $\rS_\Gamma = \Gamma \backslash \rS$.

\subsection{$\Gamma$-conjugacy classes of $\mathbb{Q}$-Parabolic Subgroups}\label{standard}
in this subsection, we determine the $\Gamma$-conjugacy classes of $\mathbb{Q}$-Parabolic Subgroups $\cP_\Q(G_2, G_2(\Z))$, which should be well-known. But as we don't know of a proper reference, a proof is given below.
\begin{lema}\label{lem-parabolic-conjugate}
  $\cP_\Q(G_2, G_2(\Z))=\{\rP_0, \rP_1, \rP_2\}$.
\end{lema}

\begin{proof}
We have to show that for each standard parabolic $\rP$, $G_2(\Z)$ acts transitively on $G_2/\rP(\Q)$. 
  Since $G_2$ is semisimple algebraic group split over $\Q$, by strong approximation, we have
  $$G_2(\Q)\prod_{p} G_2(\Z_p)=G_2(\mathbb{A}_f),$$
  which implies that $G_2(\Q)/G_2(\Z)=G_2(\mathbb{A}_f)/\prod_{p}G_2(\Z_p)$.
  Moreover, as $\Q$ is of class number $1$, we have
  $$\rP(\Q)\prod_{p} \rP(\Z_p)=\rP(\Q_p).$$
  According to the Iwasawa decomposition for $p$-adic groups, $\rP(\Q_p)G_2(\Z_p)=G_2(\Q_p)$.
  Consequently, we have $$ \rP(\Q) \prod_{p}G_2(\Z_p)=G_2(\mathbb{A}_f).$$
  Hence $\rP(\Q)$ acts transitively on $G_2(\Q)/G_2(\Z)$, which implies $G_2(\Z)$ acts transitively on $G_2(\Q)/\rP(\Q)$.

  On the other hand, we have the following sequence of Galois cohomology,
  \begin{equation*}
    1\rightarrow \rP(\Q) \rightarrow G_2(\Q)\rightarrow G_2/\rP(\Q)\rightarrow H^1(\Q, \rP)\rightarrow H^1(\Q, G_2)\rightarrow \cdots.
  \end{equation*}
  Note that $\rP$ is the semi-direct product of its unipotent radical and its Levi subgroup, which is isomorphic to either $\mathbb{G}_m^2$ or $\rGL_2$.
 According to Hilbert's Theorem 90 (see, e.g.~\cite[Chapter III]{Serre1997}), we have $H^1(\Q, \rP)={1}$. Thus $G_2/\rP(\Q)=G_2(\Q)/\rP(\Q)$, from which the lemma follows.
\end{proof}

\subsection{Irreducible Representations}\label{irreducible}
 The fundamental weights associated to $\Phi^+$  are given by $\gamma_1=2\alpha_1+ \alpha_2$ and $\gamma_2=3\alpha_1+2\alpha_2$.
 Thus irreducible finite dimensional representations of $G_2$ are determined by their highest weights, which in this case are the linear functionals of the form $m_1 \gamma_1 + m_2 \gamma_2$ with $m_1, m_2$ non-negative integers. For any $\lambda=m_1\gamma_1+m_2\gamma_2$, we set $\m_\lambda$ to be the representation defined over $\Q$ with highest weight $\lambda$.

\subsection{Kostant Representatives} \label{KostantRep}
It is known that the Weyl group $\mathcal{W}=\mathcal{W}(\Phi)$ is the dihedral group $D_6$ given by $12$ elements. They are listed in the first column of Table~\ref{Weylgp} and described in the second column as a product of simple reflections $s_{1}$ and $s_2$, associated to the simple roots $\alpha_1$ and $\alpha_2$ respectively. Then we have
\begin{eqnarray}
s_1(\alpha_1)=-\alpha_1,\  s_1(\alpha_2)=3\alpha_1+\alpha_2 \quad\text{ and }\quad s_2(\alpha_1)=\alpha_1+\alpha_2, \ s_2(\alpha_2)=-\alpha_2.
\end{eqnarray}

In the third column we make a note of their lengths and in the last column we describe the element $w \cdot \lambda = w(\lambda + \rho) - \rho$, where the pair $(a, b)$ denotes the element $a\alpha_1 + b\alpha_2 \in \ft^\ast$.

\begin{table}[ht]
\centering
\begin{tabular}{clcl}
\hline \noalign{\smallskip}
Label & $w$  & $\ell(w)$ & $w \cdot \lambda$\\
\noalign{\smallskip}\hline\noalign{\smallskip}
$w_1$ & $1$ & $0$ & $(2m_1+3m_2, \ m_1+2m_2)$\\
$w_2$ & $s_1$ & $1$ & $(m_1+3m_2-1,\  m_1+2m_2)$\\
$w_3$ & $s_{2}$ & $1$ & $(2m_1+3m_2,\  m_1+m_2-1)$\\
$w_4$ & $s_{1} s_{2}$ & $2$ & $(m_1-4,\  m_1+m_2-1)$\\
$w_5$ & $s_{2} s_{1}$  & $2$ & $(m_1+3m_2-1,\  m_2-2)$\\
$w_6$ & $s_{1} s_{2} s_{1}$ & $3$ & $(-m_1-6,\  m_2-2)$\\
$w_7$ & $s_{2}  s_{1} s_{2}$  & $3$ & $(m_1-4,\  -m_2-4)$ \\
$w_8$ & $s_{1} s_{2} s_{1} s_{2}$  & $4$ & $(-m_1-3m_2-9,\  -m_2-4)$\\
$w_9$ & $s_{2}  s_{1} s_{2} s_{1}$  & $4$ & $(-m_1-6,\  -m_1-m_2-5)$\\
$w_{10}$ & $s_{1} s_{2} s_{1} s_{2} s_{1}$  & $5$ & $(-2m_1-3m_2-10,\  -m_1-m_2-5)$\\
$w_{11}$ & $s_{2}  s_{1} s_{2} s_{1} s_{2}$  & $5$ & $(-m_1-3m_2-9,\  -m_1-2m_2-6)$\\
$w_{12}$ & $s_{1} s_{2} s_{1} s_{2} s_{1} s_{2}$  & $6$ & $(-2m_1-3m_2-10,\  -m_1-2m_2-6)$\\
\noalign{\smallskip}\hline
\end{tabular}
\vspace{0.4cm}
\caption{The Weyl Group of $G_2$}\label{Weylgp}
\end{table}

The Weyl group acts naturally on the set of roots. For each $i \in \left\{0, 1, 2\right\}$, let $\Delta(\mathfrak{u}_i)$ denote the set consisting of every root whose corresponding root space is contained in the Lie algebra  $\mathfrak{u}_i$ of the unipotent radical of $\rP_i$. The set of Weyl representatives $\W^{\rP_i} \subset \W$ associated to the parabolic subgroup $\rP_i$ (see~\cite{Kostant61}) is defined by
\[
	\mathcal{W}^{\rP_i} = \left\{w \in \mathcal{W} : w(\Phi^-) \cap \Phi^+ \subset \Delta(\mathfrak{u}_i)\right\}.
\]

Clearly $\W^{\rP_0}=\W$ and, by using the table, one can see that
\begin{eqnarray}\label{eq:Weyl12}
\W^{\rP_1}  &= \left\{ 1, s_2, s_2s_1, s_2s_1s_2, s_2s_1s_2s_1, s_2s_1s_2s_1s_2 \right\} \\ \W^{\rP_2} & =  \left\{1, s_1, s_1s_2, s_1s_2s_1, s_1s_2s_1s_2, s_1s_2s_1s_2s_1 \right\}  \,.
\end{eqnarray}


\subsubsection{Kostant representatives for minimal parabolic $\mr{P}_0$}
\begin{align} \label{highest0}
w_1\cdot\lambda &=  m_1\gamma_1 + m_2 \gamma_2 \nonumber \\
w_2\cdot\lambda &=  (-m_1-2) \gamma_1 + (m_1+m_2+1) \gamma_2 \nonumber \\
w_3\cdot\lambda &= (m_1+3m_2+3) \gamma_1 + (-m_2-2) \gamma_2 \nonumber \\
w_4\cdot\lambda &=  (-m_1-3m_2-5) \gamma_1 + (m_1+2m_2+2) \gamma_2  \nonumber \\
w_5\cdot\lambda &= (2m_1+3m_2+4) \gamma_1 + (-m_1-m_2-3) \gamma_2 \nonumber \\
w_6\cdot\lambda &= (-2m_1-3m_2-6) \gamma_1 + (m_1+2m_2+2) \gamma_2  \nonumber \\
w_7 \cdot \lambda &= (2m_1+3m_2+4) \gamma_1 + (-m_1-2m_2-4) \gamma_2 \nonumber \\
w_8\cdot\lambda &=  (-2m_1-3m_2-6) \gamma_1 + (m_1+m_2+1) \gamma_2 \nonumber \\
w_9\cdot\lambda &= (m_1+3m_2+3) \gamma_1 + (-m_1-2m_2-4) \gamma_2 \nonumber \\
w_{10}\cdot\lambda &= (-m_1-3m_2-5) \gamma_1 + m_2 \gamma_2  \nonumber \\
w_{11}\cdot\lambda &= m_1 \gamma_1 + (-m_1-m_2-3) \gamma_2 \nonumber \\
w_{12}\cdot\lambda &= (-m_1-2) \gamma_1 + (-m_2-2) \gamma_2  \nonumber
\end{align}

\subsubsection{Kostant representatives for maximal parabolic $\mr{P}_1$}
Here we take $$\gamma^{\rM_1}=\frac{1}{2}\alpha_1=\gamma_1-\frac{1}{2}\gamma_2 \quad \text{ and }\quad \kappa^{\rM_1}=\frac{1}{2}\gamma_2.$$
\begin{align}
w_1\cdot\lambda &=  m_1 \gamma^{\rM_1} + (m_1+2m_2) \kappa^{\rM_1} \nonumber \\
w_3\cdot\lambda &=  (m_1+3m_2+3) \gamma^{\rM_1} + (m_1+m_2-1) \kappa^{\rM_1} \nonumber \\
w_5\cdot\lambda &=  (2m_1+3m_2+4) \gamma^{\rM_1} + (m_2-2) \kappa^{\rM_1}\nonumber \\
w_7\cdot\lambda &=  (2m_1+3m_2+4) \gamma^{\rM_1} + (-m_2-4) \kappa^{\rM_1} \nonumber \\
w_9\cdot\lambda &=  (m_1+3m_2+3) \gamma^{\rM_1} + (-m_1-m_2-5) \kappa^{\rM_1}  \nonumber \\
w_{11} \cdot \lambda & = m_1 \gamma^{\rM_1} + (-m_1-2m_2-6) \kappa^{\rM_1}  \nonumber
\end{align}

\subsubsection{Kostant representatives for maximal parabolic $\mr{P}_2$}
Here we take $$\gamma^{\rM_2}=\frac{1}{2}\alpha_2=-\frac{3}{2}\gamma_1+\gamma_2 \quad \text{ and } \quad \kappa^{\rM_1}=\frac{1}{2}\gamma_1$$
\begin{align}
w_1\cdot\lambda &=  m_2 \gamma^{\rM_2} + (2m_1+3m_2) \kappa^{\rM_2} \nonumber \\
w_2\cdot\lambda &= (m_1+m_2+1) \gamma^{\rM_2} + (m_1+3m_2-1) \kappa^{\rM_2} \nonumber \\
w_4\cdot\lambda &= (m_1+2m_2+2) \gamma^{\rM_2} + (m_1-4)\kappa^{\rM_2}  \nonumber \\
w_6\cdot\lambda &= (m_1+2m_2+2) \gamma^{\rM_2} + (-m_1-6)\kappa^{\rM_2} \nonumber \\
w_8\cdot\lambda &=  (m_1+m_2+1) \gamma^{\rM_2} + (-m_1-3m_2-9)\kappa^{\rM_2} \nonumber \\
w_{10} \cdot \lambda & = m_2 \gamma^{\rM_2} + (-2m_1-3m_2-10)\kappa^{\rM_2} \nonumber
\end{align}

For $i=1,2$, we also write
\begin{equation}\label{def-a-b}
w\cdot \lambda= a_i(\lambda, w) \gamma^{\rM_i}+ b_i(\lambda, w) \kappa^{\rM_i}.
\end{equation}

The symmetry of the coefficients can be explained by the following lemma.

\begin{lema}\label{lem-involution}
Let $\rP$ be a parabolic subgroup of $G_2$ with Levi subgroup $\rM$,  $\rA_\rP$  be the central torus of $\rM$ and $\rS_\rP$ be the unique maximal torus of the semisimple part of $\rM$ contained in $\rT$. Let $\rN_\rP$ denote the unipotent radical of $\rP$ in $G_2$ and $w_{G_2}$ (resp. $w_\rM$) be the longest Weyl element in $\W$ (resp. $\W_\rM$). Then the following are true.
\begin{itemize}
  \item[(a)] The map $w\mapsto w':= w_\rM w w_{G_2}$ defines an involution on $\W^\rP$ and $\ell(w)+\ell(w')=\dim \rN_\rP$.
  \item[(b)] $w'(\lambda+\rho)-\rho |_{\rA_\rP}=w(\lambda+\rho)-\rho |_{\rA_\rP}$.
  \item[(c)] $w'(\lambda+\rho)-\rho |_{\rS_\rP}+w'(\lambda+\rho)-\rho |_{\rS_\rP}=-2\rho |_{\rS_\rP}$.
\end{itemize}
\end{lema}

\begin{proof}
The proof is the same as in Schwermer \cite[Section 4.2]{Sch95} by noticing the fact that the Weyl group $\W$ is self-dual, namely $w_{G_2}=-1$.
\end{proof}


\subsection{Boundary of the Borel-Serre compactification}

In general, the boundary of the Borel-Serre compactification $\partial \rS_\Gamma = \overline{\rS}_\Gamma \setminus \rS_\Gamma$ is a finite union of the boundary components $\partial_{\rP}$ for $\rP \in \mathcal{P}_\Q(\rG, \Gamma)$, i.e.
\begin{equation}\label{eq:bdface}
\partial \rS_\Gamma = \bigcup_{\rP \in \mathcal{P}_\Q(\rG, \Gamma)} \partial_{\rP} ,
\end{equation}
where
\[\partial_\rP= (\Gamma\cap \rP(\R))\backslash\rP(\R)/ (\rP(\R) \cap K_\infty). \]

This decomposition \eqref{eq:bdface} determines a spectral sequence in cohomology abutting to the cohomology of the boundary
\begin{equation}\label{eq:SpesSeqIntro}
E^{p, q}_{1} = \bigoplus_{prk(\rP)=p+1} H^q(\partial_{\rP}, \tm_\lambda) \Rightarrow H^{p+q}(\partial \rS_\Gamma, \tm_\lambda)
\end{equation}
where $prk(\rP)$ denotes the parabolic rank of $\rP$ (the dimension of the maximal $\Q$-split torus in the center of the Levi quotient $\rM$ of $\rP$).

In our case, as the $\Q$-rank of $\rG$ is $2$,  this spectral sequence is simplified to a long exact sequence (of Mayer-Vietoris) in cohomology  of the following form
\begin{equation} \label{LongExactSequance}
\cdots \rightarrow H^{q-1}(\partial_{\mathrm{P}_0}, \tm_\lambda) \rightarrow  H^q(\partial \rS_\Gamma, \tm_\lambda) \rightarrow H^q(\partial_{\mathrm{P}_1}, \tm_\lambda)  \oplus H^q(\partial_{\mathrm{P}_2}, \tm_\lambda) \rightarrow H^q(\partial_{\mathrm{P}_0}, \tm_\lambda) \rightarrow \cdots
\end{equation}
by noticing Lemma \ref{lem-parabolic-conjugate}.
We will use this long exact sequence to describe the cohomology of the boundary of the Borel-Serre compactification. 

\subsection{The theorem of Kostant}

Let $\rP=\rM\rN$ be the decomposition into its Levi subgroup $\rM$ and unipotent radical $\rN$. Then, $\partial_\rP$ is a fibre bundle over
\[\rS_\Gamma^{\rM}= (\Gamma\cap \rM(\R))\backslash\rM(\R)/ (\rM(\R) \cap K_\infty)\]
with fibres isomorphic to $\rN_\Gamma:=(\Gamma\cap \rN(\R))\backslash \rN(\R)$. Hence we have
\[H^\bullet (\partial_{\rP}, \tm_\lambda)= H^\bullet (\rS_\Gamma^{\rM}, H^\bullet(\rN_\Gamma, \tm_\lambda)  ),\]
here we use $H^\bullet(\rN_\Gamma, \tm_\lambda)$ to denote the corresponding sheaf by abuse of notation.
By the theorem of Nomizu \cite{Nomizu}, we know
\[H^\bullet(\rN_\Gamma, \tm_\lambda) = H^\bullet(\fn, \m_\lambda). \]
By the theorem of Kostant \cite{Kostant61}, we get
\[H^{q}(\fn, \m_\lambda)=\bigoplus_{w\in \W^{\rP}:\ell(w)=q}\cN_{w\cdot \lambda},\]
where $\cN_{w\cdot \lambda}$ denotes the irreducible representation of $\rM$ with highest weight $w\cdot \lambda$. In conclusion, we get
\[H^{q}(\partial_\rP, \tm_{\lambda}) = \bigoplus_{w\in  \mathcal{W}^{\rP}} H^{q-\ell(w)}(\rS_\G^{\rm{M}}, \widetilde{\cN}_{w\cdot\lambda}).\]
Note that, when there is no ambiguity, we also use the old notation $\m_{w\cdot\lambda}$ to denote $\cN_{w\cdot\lambda}$.

\subsection{Cohomological dimension}

For any discrete group $H$, set the \emph{virtual cohomological dimension} of $H$, denoted as $\vcd H$, to be
$$\vcd H= \min \{\cd H': [H: H']<\infty \}, $$
where $\cd H'$ refers to the cohomological dimension of $H'$. Now, using the compactification they had introduced, Borel and Serre showed in~\cite{BoSe73} that for  any semisimple group $\rG$ and its arithmetic subgroup $H$:
$$\vcd H= \dim \rG- \dim K -\rank_\Q \rG,$$
where $K$ is the maximal compact subgroup of $\rG(\R)$. In particular, $\dim G_2=14$, $\dim \mathrm{SO}_4(\R)=6$ and $\rank_\Q G_2=2$, thus we have $\vcd G_2(\Z)=6$. As a consequence, $H^q( \rS_\Gamma, \tm)=0$ for all $q>6$ for any coefficient system $\m$ of $\G$.

\begin{rmk}\rm$\;$ It is interesting to note that, for any $\Q$-split semisimple group $\rG$, we have
$$\vcd \rG(\Z)= \dim \rN =\max_{w\in \W} \ell(w),$$
where $\rN$ is the unipotent radical of any Borel subgroup and $\ell(w)$ is the length of $w$. Indeed, the first equality follows from the Iwasawa decomposition, while the second one follows from theory of Weyl groups. Note that this does not hold in general. For example: when $\rG$ is $\Q$-anisotropic, $\vcd \rG(\Z)$ equals to $\dim \rG- \dim K$, which is not the maximal length of Weyl elements.
 \end{rmk}
\section{Cohomology of the boundary components}\label{parity}
The cohomology of the boundary is obtained by using a spectral sequence whose terms are expressed by the cohomology of the faces associated to each standard parabolic subgroup. In this section we establish, for each standard parabolic $\rP$ and irreducible representation $\mathcal{M}_{\nu}$ of the Levi subgroup $\mr{M} \subset \rP$ with highest weight $\nu$, a condition to be satisfied in order to have nontrivial cohomology $H^\bullet(\rS_\Gamma^\mr{M}, \widetilde{\mathcal{M}}_{\nu})$.
Here  $\tm_{\nu}$ is the sheaf on $\rS_\Gamma^\mr{M}$ given by $\m_{\nu}$.
\subsection{Minimal Parabolic Subgroup}

We analyze the parity condition imposed to the face associated to the minimal parabolic $\partial_{\rP_0}$. As mentioned, the Levi subgroup of $\rP_0$ is the two dimensional torus $\mr{M}_0 \cong \mathbb{G}_m^2$. The elements lying in $\Xi:=\mr{M}_0(\Z)\cap \rK_{\infty}$ must act trivially on the representation $\m_{\nu}$ in order to have nonzero cohomology. By using this fact one can deduce the following

\begin{lema} \label{parity0}
Let $\nu$ be given by $m'_1\gamma_1 + m'_2\gamma_2$. If $m'_1$ or $m'_2$ is odd then the corresponding local system $\tm_{\nu}$ in $\rS_\G^{\rm{M}_0}$ is cohomological trivial, i.e. $H^\bullet(\rS_\G^{\rm{M}_0}, \tm_{\nu})=0$.
\end{lema}

\begin{proof}
  According to \cite[Prop 4.3]{HR2020}, we have
  \begin{equation*}
   H^\bullet(\rS_\G^{\rm{M}_0}, \tm_{\nu})=H^\bullet(\widetilde{\rS_\G^{\rm{M}_0}}, \tm_{\nu})^{\Xi},
  \end{equation*}
  where $\widetilde{\rS_\G^{\rm{M}_0}}$ denotes the locally symmetric spaces associated to $\rM_0$. As $\widetilde{\rS_\G^{\rm{M}_0}}$ is simply a point, all the higher cohomology vanishes. It is clear that $\Xi\cong (\Z/2\Z)^2$ and for each $\xi\in \Xi$ the action on the $\m$ is given by $\nu(\xi)\in \{-1, 1\}$. Thus, if there exists $\xi\in \Xi$ with $\nu(\xi)=-1$, then $H^0(\rS_\G^{\rm{M}_0}, \tm_{\nu})=0$. On the other hand, since $\gamma_1, \gamma_2$ forms an integral basis for $X^*(\rM_0):= \mathrm{Hom} (\rM_0, \mathbb{G}_m)$, there exists $\xi\in \Xi$ with $\nu(\xi)=-1$ if $m'_1$ or $m'_2$ is odd. This completes the proof.
\end{proof}
Note that every $\nu$ will be of the form $w\cdot \lambda$ for $w\in \W$. We denote by $\overline{\W}^{0}$ the set of Weyl elements $w$ such that $w\cdot \lambda$ do not satisfy the condition of Lemma~\ref{parity0}.

\begin{rmk}\rm$\;$
For notational convenience, we simply use $\partial_{i}$ to denote the boundary face $\partial_{\rP_{i}}$ associated to the parabolic subgroup $\rP_{i}$ and the arithmetic group $\Gamma$ for $i\in\{0,1,2\}$.
\end{rmk}

\subsubsection{Cohomology groups of $\partial_0$}
In this case $H^{q}(\rS_\G^{\rm{M}_0}, \widetilde{\m}_{w\cdot\lambda})=0$ for every $q\geq 1$. The Weyl group $\W^{\rP_{0}}= \mathcal{W}$ and the lengths of its  elements  are between 0 and 6 as shown in the Table~\ref{Weylgp} above. We know that
\begin{eqnarray}
 H^{q}(\partial_0, \tm_{\lambda}) &= &\bigoplus_{w\in  \mathcal{W}^{\rP_{0}}} H^{q-\ell(w)}(\rS_\G^{\rm{M}_0}, \widetilde{\m}_{w\cdot\lambda})\\
 &=& \bigoplus_{w\in  \mathcal{W}^{\rP_{0}} : \ell(w)=q} H^{0}(\rS_\G^{\rm{M}_0}, \widetilde{\m}_{w\cdot\lambda})  \nonumber
 \end{eqnarray}
 Therefore
 \begin{eqnarray}
 H^{0}(\partial_0, \tm_{\lambda}) &= & H^{0}(\rS_\G^{\rm{M}_0}, \widetilde{\m}_\lambda)\nonumber\\
 H^{1}(\partial_0, \tm_{\lambda}) &= & H^{0}(\rS_\G^{\rm{M}_0}, \widetilde{\m}_{s_1\cdot\lambda}) \oplus H^{0}(\rS_\G^{\rm{M}_0}, \widetilde{\m}_{s_2\cdot\lambda}) \nonumber\\
H^{2}(\partial_0, \tm_{\lambda}) &= & H^{0}(\rS_\G^{\rm{M}_0}, \widetilde{\m}_{s_1s_2\cdot\lambda}) \oplus H^{0}(\rS_\G^{\rm{M}_0}, \widetilde{\m}_{s_2s_1\cdot\lambda}) \nonumber\\
H^{3}(\partial_0, \tm_{\lambda}) &= & H^{0}(\rS_\G^{\rm{M}_0}, \widetilde{\m}_{s_1s_2s_1\cdot\lambda}) \oplus H^{0}(\rS_\G^{\rm{M}_0}, \widetilde{\m}_{s_2s_1s_2\cdot\lambda})  \nonumber\\
H^{4}(\partial_0, \tm_{\lambda}) &= & H^{0}(\rS_\G^{\rm{M}_0}, \widetilde{\m}_{s_1s_2s_1s_2\cdot\lambda})  \oplus H^{0}(\rS_\G^{\rm{M}_0}, \widetilde{\m}_{s_2s_1s_2s_1\cdot\lambda}) \nonumber\\
H^{5}(\partial_0, \tm_{\lambda}) &= & H^{0}(\rS_\G^{\rm{M}_0}, \widetilde{\m}_{s_1s_2s_1s_2s_1\cdot\lambda})  \oplus H^{0}(\rS_\G^{\rm{M}_0}, \widetilde{\m}_{s_2s_1s_2s_1s_2\cdot\lambda}) \nonumber\\
H^{6}(\partial_0, \tm_{\lambda}) &= & H^{0}(\rS_\G^{\rm{M}_0}, \widetilde{\m}_{s_1s_2s_1s_2s_1s_2\cdot\lambda}) \nonumber
\end{eqnarray}
 and for every $q \geq 7$, the cohomology groups $H^{q}(\partial_0, \tm_{\lambda})=0$.

\subsection{Maximal Parabolic Subgroups}

In this section we study the parity conditions for the maximal parabolics. Let $i \in \left\{1, 2\right\}$, then $\rm{M}_i \cong \mr{GL}_2$ and in this setting, $K_\infty\cap \rM_i(\R)=\mr{O}_2(\R)$ is the orthogonal group and $\G_{\rm{M}_i} = \mr{GL}_2(\Z)$. Therefore
\[
	\mr{S}_\G^{\rm{M}_i}  = \mr{GL}_2(\mathbb{Z}) \backslash \mr{GL}_2(\mathbb{R})/\mr{O}(2) \mathbb{R}_{>0}^\times\,.
\]

We also consider the following double cover of $\mr{S}_\Gamma^{\rm{M}_i}$,
\[
	\widetilde{\mr{S}_\G^{\rm{M}_i}}  =\mr{GL}_2(\mathbb{Z}) \backslash  \mr{GL}_2(\mathbb{R})/\mr{SO}(2)\mathbb{R}_{>0}^\times,
\]
which is isomorphic  to the locally symmetric space associated to $\rm{M}_i$.

Let $i$ be $1$ or $2$. Recall from \eqref{def-a-b} that, for $w \in \mathcal{W}^{\rP_i}$,  $w\cdot \lambda=a_i(\lambda, w) \gamma^{\mr{M}_i} + b_i(\lambda, w) \kappa^{\mr{M}_i}$.

\begin{lema}\label{parity1}
 If $a_i(\lambda, w)$ is odd, or equivalently, $b_i(\lambda, w)$ is odd as they are congruent modulo $2$, we have  $H^\bullet(\rS_\G^{\mr{M}_{i}}, \tm_{w\cdot\lambda})=0$.  Moreover, if $a_i(\lambda, w) = 0$ and $b_i(\lambda, w)/2$ is odd, then $H^\bullet(\rS_\G^{\mr{M}_{i}}, \tm_{w\cdot\lambda})=0$.
\end{lema}

\begin{proof} The proof here follows the proof of Lemma~\ref{parity0} closely.  According to~\cite[Prop 4.3]{HR2020}, we have
  \begin{equation*}
   H^\bullet(\rS_\G^{\rm{M}_i}, \tm_{w\cdot\lambda})=H^\bullet(\widetilde{\rS_\G^{\rm{M}_i}}, \tm_{w\cdot\lambda})^{\Xi_i},
  \end{equation*}
  where $\Xi_i=\rM_i(\Z)\cap K_{\infty}$. Note that $\rM_i$ can be identified with $\mr{GL}_2$, and $K_{\infty}\cap \rM_i(\R)$ equals to $\mr{O}_2(\R)$, hence $\Xi_i$ can be identified as $\mr{GL}_2(\Z) \cap \mr{O}_2(\R)$. For the element
  \[
\left( \begin{array}{cc}
-1 & 0 \\
0 & -1 \\
\end{array}  \right) \in \Xi_i,
\]
 the action on $H^\bullet(\widetilde{\rS_\G^{\rm{M}_i}}, \tm_{\nu})$ is given by $(-1)^{a_i(\lambda, w)}$, hence we get the first conclusion. On the other hand, consider the element
 \[
\left( \begin{array}{cc}
-1 & 0 \\
0 & 1 \\
\end{array}  \right) \in \Xi_i,
\]
whose action on $H^\bullet(\widetilde{\rS_\G^{\rm{M}_i}}, \tm_{\nu})$ is given by $(-1)^{b_i(\lambda, w)/2}$ when $a=0$. This completes the proof.
\end{proof}

Note that, as a representation of $\rM_i \cong \rGL_2 (i=1,2)$,
\begin{equation}\label{iso-gl2}
\m_{w\cdot \lambda}\cong Sym^{a_i(\lambda, w)} V \otimes Det^{\frac{b_i(\lambda,w)- a_i(\lambda, w)}{2}},
\end{equation}
where $V$ denotes the standard representation of $\rGL_2$. Let $\rB\subset \rM_i$ be a standard Borel subgroup and $\rN$ be its unipotent radical with $\fn$ its Lie algebra. Then for $i=1,2$, we have the following exact sequence,
\[H^0(\rS_\Gamma^{\rM_0}, H^0(\fn, \m_{w\cdot \lambda}))\hookrightarrow H^1_c(\rS_\Gamma^{\rM_i},\m_{w\cdot \lambda})\rightarrow H^1(\rS_\Gamma^{\rM_i},\m_{w\cdot \lambda}) \twoheadrightarrow H^0(\rS_\Gamma^{\rM_0}, H^1(\fn, \m_{w\cdot \lambda})).\]
Here, by abuse of notation, we use $H^0(\fn, \m_{w\cdot \lambda})$ and $H^1(\fn, \m_{w\cdot \lambda})$ to denote the corresponding sheaves.
In view of  Lemma \ref{parity0} and \eqref{iso-gl2}, we get
\begin{eqnarray*}
  H^0(\rS_\Gamma^{\rM_0}, H^1(\fn, \m_{w\cdot \lambda})) &=& 0, \quad \text{if }\quad \frac{b_i(\lambda,w)- a_i(\lambda, w)}{2}=0 \mod 2, \\
  H^0(\rS_\Gamma^{\rM_0}, H^0(\fn, \m_{w\cdot \lambda})) &=& 0, \quad \text{if } \quad \frac{b_i(\lambda,w)- a_i(\lambda, w)}{2}\ne 0 \mod 2.
\end{eqnarray*}
Consequently, we have
\begin{eqnarray}
  H^1_!(\rS_\Gamma^{\rM_i},\m_{w\cdot \lambda}) &=& H^1(\rS_\Gamma^{\rM_i},\m_{w\cdot \lambda}), \quad \text{if }\quad \frac{b_i(\lambda,w)- a_i(\lambda, w)}{2}=0 \mod 2, \label{parity2}\\
  H^1_!(\rS_\Gamma^{\rM_i},\m_{w\cdot \lambda}) &=& H^1_c(\rS_\Gamma^{\rM_i},\m_{w\cdot \lambda}), \quad \text{if } \quad \frac{b_i(\lambda,w)- a_i(\lambda, w)}{2}\ne 0 \mod 2. \label{parity3}
\end{eqnarray}
This fact will be extensively used in the next section for the computation of boundary cohomology.

Throughout the paper, we set
\begin{equation}\label{def-cuspforms}
S_{m+2}=H^1_!(\rGL_2, Sym^m V)\cong H^1_!(\rGL_2, Sym^m V\otimes Det).
\end{equation}
It is well-known that $H^\bullet_!(\rGL_2, \m)= H^\bullet_{cusp}(\rGL_2, \m)$ for finite dimensional representation $\m$, hence it is safe to replace $``!"$ by $``cusp"$ in \eqref{def-cuspforms}. Consequently, by Eichler-Shimura isomorphism, the space $S_{m+2}$ can be identified with the space of holomorphic cusp forms of weight $m+2$.

\begin{rmk} For notational convenience, in what follows we will denote the set of Weyl elements for which $H^\bullet(\rS_\G^{\mr{M}_{i}}, \tm_{w\cdot\lambda})\neq0$ by $\overline{\W}^{i}$. \end{rmk}

In the following subsections we make note of the cohomology groups associated to the boundary components $\partial_1$ and $\partial_2$ which will be used in the computations involved to determine the boundary cohomology in the next section.

\subsubsection{Cohomology of $\partial_1$}
In this case, the Levi $\mr{M}_1$ is isomorphic to $\mr{GL}_2$ and therefore $H^{q}(\rS_\G^{\mr{M}_1}, \tm_{w\cdot\lambda})=0$ for every $q\geq 2$. The Weyl group $\W^{\rP_1}=\{e, s_1, s_1 s_2, s_1s_2s_1, s_1s_2s_1s_2, s_1s_2s_1s_2s_1 \}$ where the length of elements are respectively $0,1,2, 3,4,5$. Thus,
\begin{eqnarray}
 H^{q}(\partial_1, \tm_{\lambda}) &= &\bigoplus_{w\in \W^{\rP_{1}}} H^{q-\ell(w)}(\rS_\G^{\rm{M}_1}, \widetilde{\m}_{w\cdot\lambda})\\
 &=& H^{q}(\rS_\G^{\rm{M}_1}, \widetilde{\m}_\lambda) \oplus H^{q-1}(\rS_\G^{\rm{M}_1}, \widetilde{\m}_{{s_1\cdot\lambda}})  \oplus \, H^{q-2}(\rS^{\rm{M}_1}, \widetilde{\m}_{{s_1s_2}\cdot\lambda})\,\nonumber \\
 &\oplus & \, H^{q-3}(\rS_\G^{\rm{M}_1}, \widetilde{\m}_{{s_1s_2s_1}\cdot\lambda})\,  \oplus \, H^{q-4}(\rS_\G^{\rm{M}_1}, \widetilde{\m}_{{s_1s_2s_1s_2}\cdot\lambda})\, \nonumber \\
 &\oplus &  \, H^{q-5}(\rS_\G^{\rm{M}_1}, \widetilde{\m}_{{s_1s_2s_1s_2s_1}\cdot\lambda})\,. \nonumber
 \end{eqnarray}
 Therefore,
 \begin{eqnarray}
 H^{0}(\partial_1, \tm_{\lambda}) & = & H^{0}(\rS_\G^{\rm{M}_1}, \widetilde{\m}_{\lambda})\nonumber\\
 H^{1}(\partial_1, \tm_{\lambda}) & = & H^{1}(\rS_\G^{\rm{M}_1}, \widetilde{\m}_{\lambda}) \oplus H^{0}(\rS_\G^{\rm{M}_1}, \widetilde{\m}_{{s_1\cdot\lambda}}) \nonumber\\
H^{2}(\partial_1, \tm_{\lambda}) & = &  H^{1}(\rS_\G^{\rm{M}_1}, \widetilde{\m}_{{s_1\cdot\lambda}}) \oplus H^{0}(\rS_\G^{\rm{M}_1}, \widetilde{\m}_{{s_1 s_2\cdot\lambda}})\nonumber\\
H^{3}(\partial_1, \tm_{\lambda}) & = &  H^{1}(\rS_\G^{\rm{M}_1}, \widetilde{\m}_{{s_1 s_2\cdot\lambda}}) \oplus H^{0}(\rS_\G^{\rm{M}_1}, \widetilde{\m}_{{s_1 s_2 s_1\cdot\lambda}})\nonumber\\
H^{4}(\partial_1, \tm_{\lambda}) & = &  H^{1}(\rS_\G^{\rm{M}_1}, \widetilde{\m}_{{s_1 s_2 s_1\cdot\lambda}}) \oplus H^{0}(\rS_\G^{\rm{M}_1}, \widetilde{\m}_{{s_1 s_2 s_1 s_2 \cdot\lambda}})\nonumber\\
H^{5}(\partial_1, \tm_{\lambda}) & = &  H^{1}(\rS_\G^{\rm{M}_1}, \widetilde{\m}_{{s_1 s_2 s_1 s_2 \cdot\lambda}}) \oplus H^{0}(\rS_\G^{\rm{M}_1}, \widetilde{\m}_{{s_1 s_2 s_1 s_2 s_1\cdot\lambda}})\nonumber\\
H^{6}(\partial_1, \tm_{\lambda}) & = &  H^{1}(\rS_\G^{\rm{M}_1}, \widetilde{\m}_{{s_1 s_2s_1s_2s_1  \cdot\lambda}}) \nonumber
\end{eqnarray}
and for every $q \geq 7$, the cohomology groups $H^{q}(\partial_1, \tm_{\lambda})=0$.

\subsubsection{Cohomology of $\partial_2$}
In this case, the Levi $\mr{M}_2$ is isomorphic to $\mr{GL}_2$ as well and therefore $H^{q}(\rS_\G^{\mr{M}_2}, \widetilde{\m}_{w\cdot\lambda})=0$ for every $q\geq 2$. The Weyl group $W^{P_2}=\{ e, s_2, s_2s_1 , s_2 s_1 s_2, s_2 s_1 s_2 s_1, s_2 s_1 s_2 s_1 s_2 \}$ where the length of elements are respectively $0,1,2, 3,4,5$. Thus,
\begin{eqnarray}
 H^{q}(\partial_2, \tm_{\lambda}) &= &\bigoplus_{w\in W^{\rP_{2}}} H^{q-\ell(w)}(\rS_\G^{\mr{M}_{2}}, \widetilde{\m}_{w\cdot\lambda})\\
 &=& H^{q}(\rS^{\mr{M}_{2}}, \widetilde{\m}_\lambda) \oplus H^{q-1}(\rS_\Gamma^{\rm{M}_2}, \widetilde{\m}_{{s_2\cdot\lambda}})  \oplus \, H^{q-2}(\rS_\G^{\mr{M}_{2}}, \widetilde{\m}_{{s_2 s_1\cdot\lambda}})\,.\nonumber \\
 &\oplus & \, H^{q-3}(\rS_\G^{\rm{M}_2}, \widetilde{\m}_{{s_2s_1s_2}\cdot\lambda})\,  \oplus \, H^{q-4}(\rS_\G^{\rm{M}_2}, \widetilde{\m}_{{s_2s_1s_2s_1}\cdot\lambda})\, \nonumber \\
 &\oplus &  \, H^{q-5}(\rS_\G^{\rm{M}_2}, \widetilde{\m}_{{s_2s_1s_2s_1s_2}\cdot\lambda})\,. \nonumber
 \end{eqnarray}
 Therefore,
 \begin{eqnarray}
 H^{0}(\partial_2, \tm_{\lambda}) & = & H^{0}(\rS_\Gamma^{\rm{M}_2}, \widetilde{\m}_{\lambda})\nonumber\\
 H^{1}(\partial_2, \tm_{\lambda}) & = & H^{1}(\rS_\Gamma^{\rm{M}_2}, \widetilde{\m}_{\lambda}) \oplus H^{0}(\rS_\Gamma^{\rm{M}_2}, \widetilde{\m}_{{s_2\cdot\lambda}}) \nonumber\\
H^{2}(\partial_1, \tm_{\lambda}) & = &  H^{1}(\rS_\Gamma^{\rm{M}_2}, \widetilde{\m}_{{s_2\cdot\lambda}}) \oplus H^{0}(\rS_\Gamma^{\rm{M}_2}, \widetilde{\m}_{{s_2 s_1\cdot\lambda}})\nonumber\\
H^{3}(\partial_1, \tm_{\lambda}) & = &  H^{1}(\rS_\Gamma^{\rm{M}_2}, \widetilde{\m}_{{s_2 s_1\cdot\lambda}}) \oplus H^{0}(\rS_\Gamma^{\rm{M}_2}, \widetilde{\m}_{{s_2 s_1 s_2 \cdot\lambda}})\nonumber\\
H^{4}(\partial_1, \tm_{\lambda}) & = &  H^{1}(\rS_\Gamma^{\rm{M}_2}, \widetilde{\m}_{{s_2 s_1 s_2 \cdot\lambda}}) \oplus H^{0}(\rS_\Gamma^{\rm{M}_2}, \widetilde{\m}_{{s_2 s_1 s_2 s_1  \cdot\lambda}})\nonumber\\
H^{5}(\partial_1, \tm_{\lambda}) & = &  H^{1}(\rS_\Gamma^{\rm{M}_2}, \widetilde{\m}_{{s_2 s_1 s_2 s_1  \cdot\lambda}}) \oplus H^{0}(\rS_\Gamma^{\rm{M}_2}, \widetilde{\m}_{{s_2 s_1 s_2 s_1 s_2 \cdot\lambda}})\nonumber\\
H^{6}(\partial_1, \tm_{\lambda}) & = &  H^{1}(\rS_\Gamma^{\rm{M}_2}, \widetilde{\m}_{{s_2 s_1 s_2s_1s_2  \cdot\lambda}}) \nonumber
\end{eqnarray}
and for every $q \geq 7$, the cohomology groups $H^{q}(\partial_1, \tm_{\lambda})=0$.


\section{Boundary cohomology}\label{boundary}
In this section, we discuss the cohomology of the boundary by giving complete description of the spectral sequence. The covering of the boundary of the Borel-Serre compactification defines a spectral sequence in cohomology.
\[
	E_1^{p, q} = \bigoplus_{prk(P)=(p+1)} H^q(\partial_\rP, \widetilde{\mathcal{M}}_\lambda) \Rightarrow H^{p+q}(\partial \rS_\Gamma, \widetilde{\mathcal{M}}_\lambda)
\]
and the nonzero terms of $E_1^{p, q}$ are for
\begin{equation}\label{eq:pq}
(p, q) \in \left\{ (i, n) | 0\leq i\leq 1, 0 \leq n \leq 6 \right\} \,.
\end{equation}

More precisely,
\begin{eqnarray}\label{eq:e12}
 E_{ 1}^{0,q} & = & \bigoplus_{i=1}^{2} \mr{H}^{q}(\partial_{i}, \tm_{\lambda}) \nonumber\\
 &= &  \bigoplus_{i=1}^{2}\bigg[ \bigoplus_{w\in \W^{\rP_{i}}} \mr{H}^{q-\ell(w)}(\mr{S}_\Gamma^{\mr{M}_{i}}, \tm_{w\cdot\lambda})\bigg]  \,,\\
 E_{ 1}^{1,q} & = & \mr{H}^{q}(\partial_{0}, \tm_{\lambda}) \nonumber\\
& = & \bigoplus_{w \in \W^{\rP_{0}}: \ell(w)=q} \mr{H}^{0}(\mr{S}_\Gamma^{\mr{M}_{0}}, \widetilde{\m}_{w\cdot\lambda})\,.\nonumber
 \end{eqnarray}

 Since $\mr{G}_{2}$  is of rank two, the spectral sequence has only two columns namely $E_{1}^{0,q}, E_{1}^{1,q}$ and to study the boundary cohomology, the task reduces to analyze the following morphisms
\begin{equation}\label{eq:spec3}
E_{1}^{0,q} \xrightarrow{d_{1}^{0,q}} E_{1}^{1,q}
\end{equation}
where  $d_{1}^{0,q}$ is the differential map and the higher differentials vanish. One has
 \begin{eqnarray*}
 E_2^{0, q} := Ker(d_{1}^{0,q}) \quad \mr{and} \quad E_2^{1, q} := Coker(d_{1}^{0,q})\,.
 \end{eqnarray*}

In addition, due to be in rank $2$ situation, the spectral sequence degenerates in degree $2$. Therefore, we can use the fact that
\begin{equation}\label{eq:hks}
	H^k(\partial \rS_\Gamma, \tm_\lambda) = \bigoplus_{p+q=k} E_2^{p, q}\,.
\end{equation}

In other words, let us now consider the short exact sequence
\begin{eqnarray}\label{eq:ses}
0 \longrightarrow E_2^{1, q-1} \longrightarrow H^q(\partial \rS_\Gamma, \tm_\lambda) \longrightarrow E_2^{0, q} \longrightarrow 0 \quad.
\end{eqnarray}

\subsection{Case 1\,: $m_1 = 0$ and $m_2 = 0$ (trivial coefficient system)}

Following Lemma~\ref{parity0} and Lemma~\ref{parity1} from Section~\ref{parity}, we get
\begin{eqnarray*}
\overline{\W}^{0} = \{w_1, w_6, w_7, w_{12} \} \,, \quad \overline{\W}^{1}=\{w_1, w_5, w_7 \} \,\quad \mr{and} \quad \overline{\W}^{2} = \{ w_1 , w_4, w_6 \} \,.
\end{eqnarray*}

By using~\eqr{e12} we record the values of $E_{1}^{0,q}$ and $E_{1}^{1,q}$ for the distinct values of $q$ below. Note that following~\eqr{pq} we know that for $q\geq 7$, $E_{1}^{i,q}=0$ for $i=1,2$.

\begin{equation}\label{eq:e110q}
E_{1}^{0,q} =\left\{\begin{array}{cccc}
&H^{0}(\rS_\Gamma^{\rm{M}_1}, \tm_{e\cdot\lambda})\oplus H^{0}(\rS_\Gamma^{\rm{M}_2}, \tm_{e\cdot\lambda}) \,, & q=0 \\
&\\
& H^{1}(\rS_\Gamma^{\rm{M}_1}, \tm_{w_5\cdot\lambda}) \oplus  H^{1}(\rS_\Gamma^{\rm{M}_2}, \tm_{w_4 \cdot\lambda})  \,, & q=3 \\
&\\
& H^{1}(\rS_\Gamma^{\rm{M}_1}, \tm_{w_7\cdot\lambda})\oplus H^{1}(\rS_\Gamma^{\rm{M}_2}, \tm_{w_6\cdot\lambda}) \,, & q=4 \\
&\\
&0  \,, & \mr{otherwise}\\
\end{array}\qquad\,, \right .
\end{equation}
and
\begin{equation}\label{eq:e111q}
E_{1}^{1,q} =\left\{\begin{array}{cccc}
& H^{0}(\rS_\Gamma^{\rm{M}_0}, \tm_{e\cdot\lambda})\,,& q=0\\
&\\
&  H^{0}(\rS_\Gamma^{\rm{M}_0}, \tm_{w_6\cdot\lambda}) \oplus H^{0}(\rS_\Gamma^{\rm{M}_0}, \tm_{w_7\cdot\lambda})  \,,& q=3\\
&\\
&  H^{0}(\rS_\Gamma^{\rm{M}_0}, \tm_{w_{12}\cdot\lambda}) \,,& q=6\\
&\\
& 0  \,,& \mr{otherwise}\\
\end{array}\qquad\,. \right .
\end{equation}

We now make a thorough analysis of~\eqr{spec3} to get the complete description of the spaces $E_{2}^{0,q}$ and $E_{2}^{1,q}$ which will give us the cohomology $H^{q}(\partial \rS_\G, \tm_{\lambda})$. We begin with $q=0$.

\subsubsection{{At the level $q=0$}} Observe that the short exact sequence~\eqr{ses} reduces  to
\begin{eqnarray*}
0 \longrightarrow H^0(\partial \rS_\Gamma, \tm_\lambda) \longrightarrow E_2^{0, 0} \longrightarrow 0 \quad.
\end{eqnarray*}
To compute $E_2^{0, 0}$,  consider the differential $d_1^{0, 0}: E_1^{0, 0} \rightarrow E_1^{1, 0}$. Following~\eqr{e110q} and ~\eqr{e111q}, we have $d_1^{0, 0}: \Q \oplus \Q \longrightarrow \Q$ and we know 
that the differential $d_1^{0, 0}$ is surjective (see~\cite{Harder87}).
Therefore
\begin{align}\label{eq:e210}
 E_2^{0, 0} := Ker(d_1^{0, 0}) = \mathbb{Q} \quad \mr{and}  \quad E_2^{1, 0}:=Coker(d_1^{0, 0}) = 0  .
\end{align}

Hence, we get $$ H^0(\partial \rS_\Gamma, \tm_\lambda) =\Q\,.$$

\subsubsection{{At the level $q=1$}} Following~\eqr{e210}, in this case, our short exact sequence~\eqr{ses} reduces to
\begin{eqnarray*}
0 \longrightarrow H^1(\partial \rS_\Gamma, \tm_\lambda) \longrightarrow E_2^{0, 1} \longrightarrow 0 \quad\,,
\end{eqnarray*}
and we need  to compute $E_2^{0, 1}$. Consider the differential $d_{1}^{0,1}: E_1^{0, 1} \longrightarrow E_1^{1, 1}$ and following~\eqr{e110q} and~\eqr{e111q}, we observe that $d_{1}^{0,1}$ is map between zero  spaces. Therefore, we obtain
\begin{eqnarray*}\label{eq:e201}
E_2^{0, 1} = 0  \qquad \mr{and} \qquad E_2^{1, 1} = 0\,.
\end{eqnarray*}
As a result, we get $$ H^1(\partial \rS_\Gamma, \tm_\lambda) =0\,.$$

\subsubsection{{At the level $q=2$}} Following the similar process as in level $q=1$, we get
\begin{eqnarray}\label{eq:e202}
E_2^{0, 2} = 0\qquad \mr{and} \qquad E_2^{1, 2} = 0\,.
\end{eqnarray}
This results into $$ H^2(\partial \rS_\Gamma, \tm_\lambda) =0\,.$$

\subsubsection{{At the level $q=3$}} Following~\eqr{e202}, in this case, the short exact sequence~\eqr{ses} reduces to
\begin{eqnarray*}
0 \longrightarrow H^3(\partial \rS_\Gamma, \tm_\lambda) \longrightarrow E_2^{0, 3} \longrightarrow 0 \quad\,,
\end{eqnarray*}
and we need  to compute $E_2^{0, 3}$. In view of \eqref{parity3}, consider the differential $d_{1}^{0,3}: E_1^{0, 3} \longrightarrow E_1^{1, 3}$ and following~\eqr{e110q} and~\eqr{e111q}, we have
 \begin{eqnarray*}
 E_2^{0, 3} &=& H^{1}_{!}(\rS_\Gamma^{\rm{M}_1}, \tm_{w_5\cdot\lambda}) \oplus  H^{1}_{!}(\rS_\Gamma^{\rm{M}_2}, \tm_{w_4 \cdot\lambda}),  \mr{and}\\ Im(d_{1}^{0,3}) &=& H^{1}_{Eis}(\rS_\Gamma^{\rm{M}_1}, \tm_{w_5\cdot\lambda}) \oplus H^{1}_{Eis}(\rS_\Gamma^{\rm{M}_1}, \tm_{w_5\cdot\lambda}) \cong \Q \oplus \Q\,.
 \end{eqnarray*}
Therefore, \begin{eqnarray}\label{eq:e213}
E_2^{0, 3} = 0\qquad \mr{and} \qquad E_2^{1, 3} = 0\,.
\end{eqnarray}
This gives us $$ H^3(\partial \rS_\Gamma, \tm_\lambda) =0\,.$$

\subsubsection{{At the level $q=4$}} Following~\eqr{e213}, in this case, the short exact sequence~\eqr{ses} reduces to
\begin{eqnarray*}
0 \longrightarrow H^4(\partial \rS_\Gamma, \tm_\lambda) \longrightarrow E_2^{0, 4} \longrightarrow 0 \quad\,,
\end{eqnarray*}
and we need  to compute $E_2^{0, 4}$. Consider the differential $d_{1}^{0,4}: E_1^{0, 4} \longrightarrow E_1^{1, 4}$ and following~\eqr{e110q} and~\eqr{e111q}, we have $ E_2^{0, 4} = H^{1}_{!}(\rS_\Gamma^{\rm{M}_1}, \tm_{w_7\cdot\lambda})\oplus H^{1}_{!}(\rS_\Gamma^{\rm{M}_2}, \tm_{w_6\cdot\lambda}) $. Since $ Im(d_{1}^{0,4}) = H^{1}_{Eis}(\rS_\Gamma^{\rm{M}_1}, \tm_{w_7\cdot\lambda})\oplus H^{1}_{Eis}(\rS_\Gamma^{\rm{M}_2}, \tm_{w_6\cdot\lambda}) = \{ 0 \}$ by \eqref{parity2}. Therefore,
\begin{eqnarray}\label{eq:e214}
E_2^{0, 4} = 0 \qquad \mr{and} \qquad E_2^{1, 4} = 0\,.
\end{eqnarray}
and we get $$ H^4(\partial \rS_\Gamma, \tm_\lambda) = 0 \,.$$

\subsubsection{{At the level $q=5$}} Following~\eqr{e214}, in this case, the short exact sequence~\eqr{ses} reduces to
\begin{eqnarray*}
0 \longrightarrow H^4(\partial \rS_\Gamma, \tm_\lambda) \longrightarrow E_2^{0, 5} \longrightarrow 0 \quad\,,
\end{eqnarray*}
and we need  to compute $E_2^{0, 5}$. Consider the differential $d_{1}^{0,5}: E_1^{0, 5} \longrightarrow E_1^{1, 5}$ and following~\eqr{e110q} and~\eqr{e111q}, we have $d_{1}^{0,5}: 0 \longrightarrow 0$. Therefore,
\begin{eqnarray}\label{eq:e215}
E_2^{0, 5} = 0 \qquad \mr{and} \qquad E_2^{1, 5} = 0\,.
\end{eqnarray}
and we get $$ H^5(\partial \rS_\Gamma, \tm_\lambda) = 0\,.$$

\subsubsection{{At the level $q=6$}} Following~\eqr{e215}, in this case, the short exact sequence~\eqr{ses} reduces to
\begin{eqnarray*}
0 \longrightarrow H^6(\partial \rS_\Gamma, \tm_\lambda) \longrightarrow E_2^{0, 6} \longrightarrow 0 \quad\,,
\end{eqnarray*}
and we need  to compute $E_2^{0, 6}$. Consider the differential $d_{1}^{0,6}: E_1^{0, 6} \longrightarrow E_1^{1, 6}$ and following~\eqr{e110q} and~\eqr{e111q}, we have $d_{1}^{0,6}: 0 \longrightarrow \Q$. Therefore,
\begin{eqnarray}\label{eq:e216}
E_2^{0, 6} = 0 \qquad \mr{and} \qquad E_2^{1, 6} = \Q\,.
\end{eqnarray}
and we get $$ H^6(\partial \rS_\Gamma, \tm_\lambda) = 0\,.$$

\subsubsection{{At the level $q=7$}} Following~\eqr{e216}, in this case, the short exact sequence~\eqr{ses} reduces to
\begin{eqnarray*}
0 \longrightarrow \Q \longrightarrow H^7(\partial \rS_\Gamma, \tm_\lambda) \longrightarrow E_2^{0, 7} \longrightarrow 0 \quad\,,
\end{eqnarray*}
and we need  to compute $E_2^{0, 7}$. Consider the differential $d_{1}^{0,7}: E_1^{0, 7} \longrightarrow E_1^{1, 7}$ and following~\eqr{e110q} and~\eqr{e111q}, we have $d_{1}^{0,7}: 0 \longrightarrow 0$. Therefore,
\begin{eqnarray*}\label{eq:e217}
E_2^{0, 7} = 0 \qquad \mr{and} \qquad E_2^{1, 7} = 0\,.
\end{eqnarray*}
and we get $$ H^7(\partial \rS_\Gamma, \tm_\lambda) = \Q\,.$$

Hence, we can summarize the above discussion as follows :
\begin{align*}
H^q(\partial \rS_\Gamma, \tm_\lambda) =\left\{\begin{array}{cccc}
&\mathbb{Q} \,, &  q=0,7 \\
&\\
&0  \,, & \mr{otherwise}\\
\end{array}\qquad\,. \right .
\end{align*}

\subsection{Case 2\, : $m_1 = 0$, $m_2 \neq 0$, $m_2$ even}\label{case2}

Following Lemma~\ref{parity0} and Lemma~\ref{parity1} from Section~\ref{parity}, we get
\begin{eqnarray*}
\overline{\W}^{0} = \{w_1, w_6, w_7, w_{12} \} \,, \quad \overline{\W}^{1}=\{w_1, w_5, w_7 \} \,\quad \mr{and} \quad \overline{\W}^{2} = \{ w_1 , w_4, w_6, w_{10} \} \,.
\end{eqnarray*}

By using~\eqr{e12} we record the values of $E_{1}^{0,q}$ and $E_{1}^{1,q}$ for the distinct values of $q$ below. Note that following~\eqr{pq} we know that for $q\geq 7$, $E_{1}^{i,q}=0$ for $i=1,2$.

\begin{equation*}\label{eq:e210q}
E_{1}^{0,q} =\left\{\begin{array}{cccc}
&H^{0}(\rS_\Gamma^{\rm{M}_1}, \tm_{e\cdot\lambda}) \,, & q=0 \\
&\\
& H^{1}(\rS_\Gamma^{\rm{M}_2}, \tm_{\lambda})  \,, & q=1 \\
&\\
& H^{1}(\rS_\Gamma^{\rm{M}_1}, \tm_{w_5\cdot\lambda})\oplus H^{1}(\rS_\Gamma^{\rm{M}_2}, \tm_{w_4\cdot\lambda}) \,, & q=3 \\
&\\
& H^{1}(\rS_\Gamma^{\rm{M}_1}, \tm_{w_7\cdot\lambda})\oplus H^{1}(\rS_\Gamma^{\rm{M}_2}, \tm_{w_6\cdot\lambda}) \,, & q=4 \\
&\\
& H^{1}(\rS_\Gamma^{\rm{M}_2}, \tm_{w_{10}\cdot\lambda})\,, & q=6 \\
&\\
&0  \,, & \mr{otherwise}\\
\end{array}\qquad\,, \right .
\end{equation*}
and
\begin{equation*}\label{eq:e211q}
E_{1}^{1,q} =\left\{\begin{array}{cccc}
& H^{0}(\rS_\Gamma^{\rm{M}_0}, \tm_{e\cdot\lambda}) \,,& q=0\\
&\\
& H^{0}(\rS_\Gamma^{\rm{M}_0}, \tm_{w_6\cdot\lambda}) \oplus H^{0}(\rS_\Gamma^{\rm{M}_0}, \tm_{w_7\cdot\lambda})  \,,& q=3\\
&\\
&  H^{0}(\rS_\Gamma^{\rm{M}_0}, \tm_{w_{12}\cdot\lambda}) \,,& q=6\\
&\\
& 0  \,,& \mr{otherwise}\\
\end{array}\qquad\,. \right .
\end{equation*}

Having a thorough analysis of~\eqr{spec3} as in previous section, we get the complete description of the spaces $E_{2}^{0,q}$ and $E_{2}^{1,q}$ which will give us the cohomology $H^{q}(\partial \rS_\Gamma, \tm_{\lambda})$ described as follows :

\begin{equation*}
H^q(\partial \rS_\Gamma, \tm_\lambda) =\left\{\begin{array}{cccc}
& H^{1}_{!}(\rS_\Gamma^{\rm{M}_2}, \tm_{\lambda})  \cong S_{m_2 + 2}\quad \,, &  q=1 \\
&\\
& H^{1}_{!}(\rS_\Gamma^{\rm{M}_1}, \tm_{w_5\cdot\lambda})\oplus H^{1}_{!}(\rS_\Gamma^{\rm{M}_2}, \tm_{w_4\cdot\lambda})\cong  S_{3m_2 + 6} \oplus S_{2m_2 + 4}\,, & q=3 \\
&\\
& H^{1}_{!}(\rS_\Gamma^{\rm{M}_1}, \tm_{w_7\cdot\lambda})\oplus H^{1}_{!}(\rS_\Gamma^{\rm{M}_2}, \tm_{w_6\cdot\lambda})\cong S_{3m_2 + 6} \oplus S_{2m_2 + 4} \,, & q=4 \\
& \\
& H^{1}_{!}(\rS_\Gamma^{\rm{M}_2}, \tm_{w_{10}\cdot\lambda})\cong  S_{m_2 + 2} \,, & q=6 \\
&\\
& 0 \,, & \mr{otherwise}\\
\end{array}\qquad\,. \right .
\end{equation*}
Here $S_k$ is defined as in \eqref{def-cuspforms}.

We conclude the above discussion as follows :
\begin{equation*}
H^q(\partial \rS_\Gamma, \tm_\lambda) =\left\{\begin{array}{cccc}
& S_{m_2 + 2}\quad \,, & q=1, 6 \\
&\\
& S_{3m_2 + 6} \oplus S_{2m_2 + 4}\,, & q=3, 4 \\
&\\
& 0\,, & \mr{otherwise}\\
\end{array}\qquad\,. \right .
\end{equation*}

\subsection{Case 3 : $m_1 = 0$, $m_2$ odd}\label{case3}

Following Lemma~\ref{parity0} and Lemma~\ref{parity1} from Section~\ref{parity}, we get
\begin{eqnarray*}
\overline{\W}^{0} = \{w_1, w_4, w_9, w_{11} \} \,, \quad \overline{\W}^{1}=\{w_3, w_9, w_{11} \} \,\quad \mr{and} \quad \overline{\W}^{2} = \{ w_2 , w_4, w_6, w_{8} \} \,.
\end{eqnarray*}
By using~\eqr{e12} we record the values of $E_{1}^{0,q}$ and $E_{1}^{1,q}$ for the distinct values of $q$ below.
\begin{equation*}\label{eq:e310q}
E_{1}^{0,q} =\left\{\begin{array}{cccc}
& H^{1}(\rS_\Gamma^{\rm{M}_1}, \tm_{ w_3 \cdot \lambda}) \oplus  H^{1}(\rS_\Gamma^{\rm{M}_2}, \tm_{ w_2 \cdot\lambda}) \,, & q=2 \\
&\\
& H^{1}(\rS_\Gamma^{\rm{M}_2}, \tm_{w_4\cdot\lambda}) \,, & q=3 \\
&\\
& H^{1}(\rS_\Gamma^{\rm{M}_2}, \tm_{w_6\cdot\lambda}) \,, & q=4 \\
&\\
& H^{1}(\rS_\Gamma^{\rm{M}_1}, \tm_{w_{9}\cdot\lambda}) \oplus H^{0}(\rS_\Gamma^{\rm{M}_1}, \tm_{w_{11}\cdot\lambda}) \oplus H^{1}(\rS_\Gamma^{\rm{M}_2}, \tm_{w_{8}\cdot\lambda}) \,, & q=5 \\
&\\
& H^{1}(\rS_\Gamma^{\rm{M}_1}, \tm_{w_{11}\cdot\lambda}) \,, & q=6 \\
&\\
&0  \,, & \mr{otherwise}\\
\end{array}\qquad\,, \right .
\end{equation*}
and
\begin{equation*}\label{eq:e311q}
E_{1}^{1,q} =\left\{\begin{array}{cccc}
& H^{0}(\rS_\Gamma^{\rm{M}_0}, \tm_{ w_2 \cdot\lambda}) \cong \Q\,,& q=1\\
&\\
& H^{0}(\rS_\Gamma^{\rm{M}_0}, \tm_{w_4\cdot\lambda})   \cong \Q \,,& q=2\\
&\\
& H^{0}(\rS_\Gamma^{\rm{M}_0}, \tm_{w_{9}\cdot\lambda}) \cong \Q \,,& q=4\\
&\\
& H^{0}(\rS_\Gamma^{\rm{M}_0}, \tm_{w_{11}\cdot\lambda}) \cong \Q \,,& q=5\\
&\\
& 0  \,,& \mr{otherwise}\\
\end{array}\qquad\,. \right .
\end{equation*}

Having a thorough analysis of~\eqr{spec3} as in previous section, we get the complete description of the spaces $E_{2}^{0,q}$ and $E_{2}^{1,q}$ which will give us the cohomology $H^{q}(\partial \rS_\Gamma, \tm_{\lambda})$ described as follows :

\begin{equation*}
H^q(\partial \rS_\Gamma, \tm_\lambda) =\left\{\begin{array}{cccc}
& H^{1}_{!}(\rS_\Gamma^{\rm{M}_1}, \tm_{w_3 \cdot \lambda}) \oplus  H^{1}_{!}(\rS_\Gamma^{\rm{M}_2}, \tm_{w_2 \cdot \lambda}) \oplus \Q  \,, &  q=2 \\
&\\
&  H^{1}_{!}(\rS_\Gamma^{\rm{M}_2}, \tm_{w_4\cdot\lambda})  \,, & q=3 \\
&\\
& H^{1}_{!}(\rS_\Gamma^{\rm{M}_2}, \tm_{w_6\cdot\lambda})  \,, & q=4 \\
& \\
&  H^{1}_{!}(\rS_\Gamma^{\rm{M}_1}, \tm_{w_9 \cdot \lambda})  \oplus H^{1}_{!}(\rS_\Gamma^{\rm{M}_2}, \tm_{w_8\cdot\lambda}) \oplus \Q \,, & q=5 \\
&\\
& 0  \,, & \mr{otherwise}\\
\end{array}\qquad\,. \right .
\end{equation*}

We conclude the above discussion as follows :
\begin{equation*}
H^q(\partial \rS_\Gamma, \tm_\lambda) =\left\{\begin{array}{cccc}
& S_{3m_2 + 5} \oplus S_{m_2 + 3} \oplus \Q \,, & q=2, 5 \\
&\\
& S_{2m_2 + 4} \,, & q=3, 4 \\
&\\
& 0  \,, & \mr{otherwise}\\
\end{array}\qquad\,. \right .
\end{equation*}

\subsection{Case 4 : $m_1 \neq 0$ even  and $m_2 = 0$}\label{case4}

Following Lemma~\ref{parity0} and Lemma~\ref{parity1} from Section~\ref{parity}, we get
\begin{eqnarray*}
\overline{\W}^{0} = \{w_1, w_6, w_7, w_{12} \} \,, \quad \overline{\W}^{1}=\{w_1, w_5, w_7, w_{11} \} \,\quad \mr{and} \quad \overline{\W}^{2} = \{ w_1 , w_4, w_6 \} \,.
\end{eqnarray*}
By using~\eqr{e12} we record the values of $E_{1}^{0,q}$ and $E_{1}^{1,q}$ for the distinct values of $q$ below.
\begin{equation*}\label{eq:e410q}
E_{1}^{0,q} =\left\{\begin{array}{cccc}
& H^{0}(\rS_\Gamma^{\rm{M}_2}, \tm_{\lambda})  \,, & q=0 \\
&\\
& H^{1}(\rS_\Gamma^{\rm{M}_1}, \tm_{\lambda}) \oplus  H^{1}(\rS_\Gamma^{\rm{M}_2}, \tm_{\lambda}) \,, & q=1 \\
&\\
& H^{1}(\rS_\Gamma^{\rm{M}_1}, \tm_{w_5\cdot\lambda})  \oplus H^{1}(\rS_\Gamma^{\rm{M}_2}, \tm_{w_4 \cdot \lambda})\,, & q=3 \\
&\\
& H^{1}(\rS_\Gamma^{\rm{M}_1}, \tm_{w_7\cdot\lambda}) \oplus H^{1}(\rS_\Gamma^{\rm{M}_2}, \tm_{w_6 \cdot \lambda})\,, & q=4 \\
&\\
& H^{1}(\rS_\Gamma^{\rm{M}_1}, \tm_{w_{11}\cdot\lambda}) \,, & q=6 \\
&\\
&0  \,, & \mr{otherwise}\\
\end{array}\qquad\,, \right .
\end{equation*}
and
\begin{equation*}\label{eq:e411q}
E_{1}^{1,q} =\left\{\begin{array}{cccc}
& H^{0}(\rS_\Gamma^{\rm{M}_0}, \tm_{\lambda}) \cong \Q\,,& q=0\\
&\\
& H^{0}(\rS_\Gamma^{\rm{M}_0}, \tm_{w_6\cdot\lambda}) \oplus  H^{0}(\rS_\Gamma^{\rm{M}_0}, \tm_{w_7\cdot\lambda})  \cong \Q \,,& q=3\\
&\\
& H^{0}(\rS_\Gamma^{\rm{M}_0}, \tm_{w_{12}\cdot \lambda}) \cong \Q\,,& q=6\\
&\\
& 0  \,,& \mr{otherwise}\\
\end{array}\qquad\,. \right .
\end{equation*}

Having a thorough analysis of~\eqr{spec3} as in previous section, we get the complete description of the spaces $E_{2}^{0,q}$ and $E_{2}^{1,q}$ which will give us the cohomology $H^{q}(\partial \rS_\Gamma, \tm_{\lambda})$ described as follows :

\begin{equation*}
H^q(\partial \rS_\Gamma, \tm_\lambda) =\left\{\begin{array}{cccc}
&  H^{1}_{!}(\rS_\Gamma^{\rm{M}_1}, \tm_{\lambda}) \,, & q=1 \\
&\\
& H^{1}_{!}(\rS_\Gamma^{\rm{M}_1}, \tm_{w_{5} \cdot \lambda}) \oplus H^{1}_{!}(\rS_\Gamma^{\rm{M}_2}, \tm_{w_4\cdot\lambda})  \,, & q=3 \\
& \\
& H^{1}_{!}(\rS_\Gamma^{\rm{M}_1}, \tm_{w_{7} \cdot \lambda}) \oplus H^{1}_{!}(\rS_\Gamma^{\rm{M}_2}, \tm_{w_6\cdot\lambda})  \,, & q=4 \\
& \\
&  H^{1}_{!}(\rS_\Gamma^{\rm{M}_1}, \tm_{w_{11} \cdot \lambda}) \,, & q=6 \\
&\\
& 0  \,, & \mr{otherwise}\\
\end{array}\qquad\,. \right .
\end{equation*}

We conclude the above discussion as follows :
\begin{equation*}
H^q(\partial \rS_\Gamma, \tm_\lambda) =\left\{\begin{array}{cccc}
&  S_{m_1 + 2}  \,, & q=1, 6 \\
& \\
&  S_{2m_1 + 6} \oplus S_{m_1 + 4}\,, & q=3, 4 \\
&\\
& 0  \,, & \mr{otherwise}\\
\end{array}\qquad\,. \right .
\end{equation*}

\subsection{Case 5 : $m_1 (\neq 0)$ even, $m_2 (\neq 0)$ even}\label{case5}

Following Lemma~\ref{parity0} and Lemma~\ref{parity1} from Section~\ref{parity}, we get
\begin{eqnarray*}
\overline{\W}^{0} = \{w_1, w_6, w_7, w_{12} \} \,, \quad \overline{\W}^{1}=\{w_1, w_5, w_7, w_{11} \} \,\quad \mr{and} \quad \overline{\W}^{2} = \{ w_1 , w_4, w_6, w_{10} \} \,.
\end{eqnarray*}

By using~\eqr{e12} we record the values of $E_{1}^{0,q}$ and $E_{1}^{1,q}$ for the distinct values of $q$ below.

\begin{equation*}\label{eq:e510q}
E_{1}^{0,q} =\left\{\begin{array}{cccc}
& H^{1}(\rS_\Gamma^{\rm{M}_1}, \tm_{\lambda}) \oplus  H^{1}(\rS_\Gamma^{\rm{M}_2}, \tm_{\lambda}) \,, & q=1 \\
&\\
& H^{1}(\rS_\Gamma^{\rm{M}_1}, \tm_{w_5\cdot\lambda})  \oplus H^{1}(\rS_\Gamma^{\rm{M}_2}, \tm_{w_4 \cdot \lambda})\,, & q=3 \\
&\\
& H^{1}(\rS_\Gamma^{\rm{M}_1}, \tm_{w_7\cdot\lambda}) \oplus H^{1}(\rS_\Gamma^{\rm{M}_2}, \tm_{w_6 \cdot \lambda})\,, & q=4 \\
&\\
& H^{1}(\rS_\Gamma^{\rm{M}_1}, \tm_{w_{11}\cdot\lambda}) \oplus H^{1}(\rS_\Gamma^{\rm{M}_2}, \tm_{w_{10} \cdot \lambda})\,, & q=6 \\
&\\
&0  \,, & \mr{otherwise}\\
\end{array}\qquad\,, \right .
\end{equation*}
and
\begin{equation*}\label{eq:e511q}
E_{1}^{1,q} =\left\{\begin{array}{cccc}
& H^{0}(\rS_\Gamma^{\rm{M}_0}, \tm_{\lambda}) \cong \Q\,,& q=0\\
&\\
& H^{0}(\rS_\Gamma^{\rm{M}_0}, \tm_{w_6\cdot\lambda}) \oplus  H^{0}(\rS_\Gamma^{\rm{M}_0}, \tm_{w_7\cdot\lambda})  \cong \Q \,,& q=3\\
&\\
& H^{0}(\rS_\Gamma^{\rm{M}_0}, \tm_{w_{12}\cdot \lambda}) \cong \Q\,,& q=6\\
&\\
& 0  \,,& \mr{otherwise}\\
\end{array}\qquad\,. \right .
\end{equation*}

Having a thorough analysis of~\eqr{spec3} as in previous section, we get the complete description of the spaces $E_{2}^{0,q}$ and $E_{2}^{1,q}$ which will give us the cohomology $H^{q}(\partial \rS_\Gamma, \tm_{\lambda})$ described as follows :

\begin{equation*}
H^q(\partial \rS_\Gamma, \tm_\lambda) =\left\{\begin{array}{cccc}
&  H^{1}_{!}(\rS_\Gamma^{\rm{M}_1}, \tm_{\lambda}) \oplus  H^{1}_{!}(\rS_\Gamma^{\rm{M}_1}, \tm_{\lambda}) \oplus \Q\,, & q=1 \\
&\\
& H^{1}_{!}(\rS_\Gamma^{\rm{M}_1}, \tm_{w_{5} \cdot \lambda}) \oplus H^{1}_{!}(\rS_\Gamma^{\rm{M}_2}, \tm_{w_4\cdot\lambda})  \,, & q=3 \\
& \\
& H^{1}_{!}(\rS_\Gamma^{\rm{M}_1}, \tm_{w_{7} \cdot \lambda}) \oplus H^{1}_{!}(\rS_\Gamma^{\rm{M}_2}, \tm_{w_6\cdot\lambda})  \,, & q=4 \\
& \\
&  H^{1}_{!}(\rS_\Gamma^{\rm{M}_1}, \tm_{w_{11} \cdot \lambda}) \oplus  H^{1}_{!}(\rS_\Gamma^{\rm{M}_2}, \tm_{w_{10} \cdot \lambda})\oplus \Q \,, & q=6 \\
&\\
& 0  \,, & \mr{otherwise}\\
\end{array}\qquad\,. \right .
\end{equation*}

We conclude the above discussion as follows :
\begin{equation*}
H^q(\partial \rS_\Gamma, \tm_\lambda) =\left\{\begin{array}{cccc}
&  S_{m_1 + 2} \oplus S_{m_2 + 2} \oplus \Q \,, & q=1, 6 \\
& \\
&  S_{2m_1 + 3m_2+ 6} \oplus S_{m_1 + 2m_2+4}\,, & q=3, 4 \\
&\\
& 0  \,,& \mr{otherwise}\\
\end{array}\qquad\,. \right .
\end{equation*}

\subsection{Case 6 : $m_1 (\neq 0)$ even, $m_2$ odd}\label{case6}

Following Lemma~\ref{parity0} and Lemma~\ref{parity1} from Section~\ref{parity}, we get
\begin{eqnarray*}
\overline{\W}^{0} = \{w_2, w_4, w_9, w_{11} \} \,, \quad \overline{\W}^{1}=\{w_1, w_3, w_9, w_{11} \} \,\quad \mr{and} \quad \overline{\W}^{2} = \{ w_2 , w_4, w_6, w_{8} \} \,.
\end{eqnarray*}
By using~\eqr{e12} we record the values of $E_{1}^{0,q}$ and $E_{1}^{1,q}$ for the distinct values of $q$ below.
\begin{equation*}\label{eq:e610q}
E_{1}^{0,q} =\left\{\begin{array}{cccc}
& H^{1}(\rS_\Gamma^{\rm{M}_1}, \tm_{\lambda}) \,, & q=1 \\
&\\
& H^{1}(\rS_\Gamma^{\rm{M}_1}, \tm_{w_3 \cdot \lambda}) \oplus H^{1}(\rS_\Gamma^{\rm{M}_2}, \tm_{w_2 \cdot \lambda})\,, & q=2 \\
&\\
&  H^{1}(\rS_\Gamma^{\rm{M}_2}, \tm_{w_4 \cdot \lambda})\,, & q=3 \\
&\\
&  H^{1}(\rS_\Gamma^{\rm{M}_2}, \tm_{w_6 \cdot \lambda})\,, & q=4 \\
&\\
& H^{1}(\rS_\Gamma^{\rm{M}_1}, \tm_{w_9 \cdot \lambda}) \oplus H^{1}(\rS_\Gamma^{\rm{M}_2}, \tm_{w_8 \cdot \lambda})\,, & q=5 \\
&\\
& H^{1}(\rS_\Gamma^{\rm{M}_1}, \tm_{w_{11}\cdot\lambda}) \,, & q=6 \\
&\\
&0  \,, & \mr{otherwise}\\
\end{array}\qquad\,, \right .
\end{equation*}
and
\begin{equation*}\label{eq:e611q}
E_{1}^{1,q} =\left\{\begin{array}{cccc}
& H^{0}(\rS_\Gamma^{\rm{M}_0}, \tm_{w_2 \cdot \lambda}) \,,& q=1\\
&\\
& H^{0}(\rS_\Gamma^{\rm{M}_0}, \tm_{w_4 \cdot \lambda}) \,,& q=2\\
&\\
& H^{0}(\rS_\Gamma^{\rm{M}_0}, \tm_{w_9\cdot\lambda})  \,,& q=4\\
&\\
& H^{0}(\rS_\Gamma^{\rm{M}_0}, \tm_{w_{11}\cdot \lambda}) \cong \Q\,,& q=5\\
&\\
& 0  \,,& \mr{otherwise}\\
\end{array}\qquad\,. \right .
\end{equation*}

Having a thorough analysis of~\eqr{spec3} as in previous section, we get the complete description of the spaces $E_{2}^{0,q}$ and $E_{2}^{1,q}$ which will give us the cohomology $H^{q}(\partial \rS_\Gamma, \tm_{\lambda})$ described as follows :
\begin{equation*}
H^q(\partial \rS_\Gamma, \tm_\lambda) =\left\{\begin{array}{cccc}
&  H^{1}_{!}(\rS_\Gamma^{\rm{M}_1}, \tm_{\lambda}) \,, & q=1 \\
&\\
&  H^{1}_{!}(\rS_\Gamma^{\rm{M}_1}, \tm_{w_3 \cdot \lambda}) \oplus  H^{1}_{!}(\rS_\Gamma^{\rm{M}_2}, \tm_{w_2 \cdot \lambda})\,, & q=2\\
&\\
& H^{1}_{!}(\rS_\Gamma^{\rm{M}_2}, \tm_{w_4\cdot\lambda})  \,, & q=3 \\
& \\
&  H^{1}_{!}(\rS_\Gamma^{\rm{M}_2}, \tm_{w_6\cdot\lambda})  \,, & q=4 \\
& \\
&  H^{1}_{!}(\rS_\Gamma^{\rm{M}_1}, \tm_{w_9 \cdot \lambda}) \oplus  H^{1}_{!}(\rS_\Gamma^{\rm{M}_2}, \tm_{w_8 \cdot \lambda})\,, & q=5\\
&\\
&  H^{1}_{!}(\rS_\Gamma^{\rm{M}_1}, \tm_{w_{11} \cdot \lambda}) \,, & q=6 \\
&\\
& 0  \,, & \mr{otherwise}\\
\end{array}\qquad\,. \right .
\end{equation*}

We conclude the above discussion as follows :
\begin{equation*}
H^q(\partial \rS_\Gamma, \tm_\lambda) =\left\{\begin{array}{cccc}
&  S_{m_1 + 2} \,, & q=1, 6 \\
& \\
&  S_{m_1 + 3m_2+ 5} \oplus S_{m_1 + m_2+3}\,, & q=2, 5 \\
&\\
&  S_{m_1 + 2m_2+ 4} \,, & q=3, 4 \\
&\\
& 0  \,, & \mr{otherwise}\\
\end{array}\qquad\,. \right .
\end{equation*}

\subsection{Case 7:  $m_1$ odd, $m_2 =0$}

Following Lemma~\ref{parity0} and Lemma~\ref{parity1} from Section~\ref{parity}, we get
\begin{eqnarray*}
\overline{\W}^{0} = \{w_3, w_5, w_8, w_{10} \} \,, \quad \overline{\W}^{2}=\{w_2, w_8, w_{10} \} \,\quad \mr{and} \quad \overline{\W}^{1} = \{ w_3 , w_5, w_7, w_{9} \} \,.
\end{eqnarray*}

Note that this is dual to case 3. By using~\eqr{e12} we record the values of $E_{1}^{0,q}$ and $E_{1}^{1,q}$ for the distinct values of $q$ below.
\begin{equation*}\label{eq:e710q}
E_{1}^{0,q} =\left\{\begin{array}{cccc}
& H^{1}(\rS_\Gamma^{\rm{M}_1}, \tm_{w_3 \cdot \lambda}) \oplus  H^{1}(\rS_\Gamma^{\rm{M}_2}, \tm_{w_2 \cdot\lambda}) \,, & q=2 \\
&\\
& H^{1}(\rS_\Gamma^{\rm{M}_1}, \tm_{w_5\cdot\lambda}) \,, & q=3 \\
&\\
& H^{1}(\rS_\Gamma^{\rm{M}_1}, \tm_{w_7\cdot\lambda}) \,, & q=5 \\
&\\
& H^{1}(\rS_\Gamma^{\rm{M}_1}, \tm_{w_{9}\cdot\lambda}) \oplus H^{1}(\rS_\Gamma^{\rm{M}_2}, \tm_{w_{8}\cdot\lambda}) \oplus H^{0}(\rS_\Gamma^{\rm{M}_2}, \tm_{w_{10}\cdot\lambda}) \,, & q=5 \\
&\\
& H^{1}(\rS_\Gamma^{\rm{M}_2}, \tm_{w_{10}\cdot\lambda}) \,, & q=6 \\
&\\
&0  \,, & \mr{otherwise}\\
\end{array}\qquad\,, \right .
\end{equation*}
and
\begin{equation*}\label{eq:e711q}
E_{1}^{1,q} =\left\{\begin{array}{cccc}
& H^{0}(\rS_\Gamma^{\rm{M}_0}, \tm_{ w_3 \cdot\lambda}) \cong \Q\,,& q=1\\
&\\
& H^{0}(\rS_\Gamma^{\rm{M}_0}, \tm_{w_5\cdot\lambda})   \cong \Q \,,& q=2\\
&\\
& H^{0}(\rS_\Gamma^{\rm{M}_0}, \tm_{w_{8}\cdot\lambda}) \cong \Q \,,& q=4\\
&\\
& H^{0}(\rS_\Gamma^{\rm{M}_0}, \tm_{w_{10}\cdot\lambda}) \cong \Q \,,& q=5\\
&\\
& 0  \,,& \mr{otherwise}\\
\end{array}\qquad\,. \right .
\end{equation*}

Having a thorough analysis of~\eqr{spec3} as in previous section, we get the complete description of the spaces $E_{2}^{0,q}$ and $E_{2}^{1,q}$ which will give us the cohomology $H^{q}(\partial \rS_\Gamma, \tm_{\lambda})$ described as follows :
\begin{equation*}
H^q(\partial \rS_\Gamma, \tm_\lambda) =\left\{\begin{array}{cccc}
& H^{1}_{!}(\rS_\Gamma^{\rm{M}_1}, \tm_{w_3 \cdot \lambda}) \oplus  H^{1}_{!}(\rS_\Gamma^{\rm{M}_2}, \tm_{w_2 \cdot \lambda}) \oplus \Q  \,, &  q=2 \\
&\\
&  H^{1}_{!}(\rS_\Gamma^{\rm{M}_1}, \tm_{w_5\cdot\lambda})  \,, & q=3 \\
&\\
& H^{1}_{!}(\rS_\Gamma^{\rm{M}_1}, \tm_{w_7\cdot\lambda})  \,, & q=4 \\
& \\
&  H^{1}_{!}(\rS_\Gamma^{\rm{M}_1}, \tm_{w_9 \cdot \lambda})  \oplus H^{1}_{!}(\rS_\Gamma^{\rm{M}_2}, \tm_{w_8\cdot\lambda}) \oplus \Q \,, & q=5 \\
&\\
& 0  \,,&  \mr{otherwise}\\
\end{array}\qquad\,. \right .
\end{equation*}

We conclude the above discussion as follows :
\begin{equation*}
H^q(\partial \rS_\Gamma, \tm_\lambda) =\left\{\begin{array}{cccc}
& S_{m_1 + 5} \oplus S_{m_1 + 3} \oplus \Q   \,, & q=2, 5 \\
&\\
& S_{2m_1 + 6} \,, & q=3, 4 \\
&\\
& 0  \,, & \mr{otherwise}\\
\end{array}\qquad\,. \right .
\end{equation*}

\subsection{Case 8 :  $m_1$ odd, $m_2 (\neq 0)$ even}

Following Lemma~\ref{parity0} and Lemma~\ref{parity1} from Section~\ref{parity}, we get
\begin{eqnarray*}
\overline{\W}^{0} = \{w_3, w_5, w_8, w_{10} \} \,, \quad \overline{\W}^{1}=\{w_3, w_5, w_7, w_{9} \} \,\quad \mr{and} \quad \overline{\W}^{2} = \{ w_1 , w_2, w_8, w_{10} \} \,.
\end{eqnarray*}

Note that this is dual to case 6. By using~\eqr{e12} we record the values of $E_{1}^{0,q}$ and $E_{1}^{1,q}$ for the distinct values of $q$ below.
\begin{equation*}\label{eq:e810q}
E_{1}^{0,q} =\left\{\begin{array}{cccc}
&  H^{1}(\rS_\Gamma^{\rm{M}_2}, \tm_{\lambda}) \,, & q=1 \\
&\\
& H^{1}(\rS_\Gamma^{\rm{M}_1}, \tm_{w_3\cdot\lambda})  \oplus H^{1}(\rS_\Gamma^{\rm{M}_2}, \tm_{w_2 \cdot \lambda})\,, & q=2 \\
&\\
& H^{1}(\rS_\Gamma^{\rm{M}_1}, \tm_{w_5\cdot\lambda}) \,, & q=3 \\
&\\
& H^{1}(\rS_\Gamma^{\rm{M}_1}, \tm_{w_7\cdot\lambda}) \,, & q=4 \\
&\\
& H^{1}(\rS_\Gamma^{\rm{M}_1}, \tm_{w_{9}\cdot\lambda}) \oplus H^{1}(\rS_\Gamma^{\rm{M}_2}, \tm_{w_{8} \cdot \lambda})\,, & q=5 \\
&\\
&  H^{1}(\rS_\Gamma^{\rm{M}_2}, \tm_{w_{10} \cdot \lambda})\,, & q=6 \\
&\\
&0  \,, & \mr{otherwise}\\
\end{array}\qquad\,, \right .
\end{equation*}
and
\begin{equation*}\label{eq:e811q}
E_{1}^{1,q} =\left\{\begin{array}{cccc}
& H^{0}(\rS_\Gamma^{\rm{M}_0}, \tm_{w_3 \cdot \lambda}) \cong \Q\,,& q=1\\
&\\
& H^{0}(\rS_\Gamma^{\rm{M}_0}, \tm_{w_5\cdot \lambda}) \cong \Q\,,& q=2\\
&\\
& H^{0}(\rS_\Gamma^{\rm{M}_0}, \tm_{w_8\cdot\lambda})  \cong \Q \,,& q=4\\
&\\
& H^{0}(\rS_\Gamma^{\rm{M}_0}, \tm_{w_{10}\cdot \lambda}) \cong \Q\,,& q=5\\
&\\
& 0  \,,& \mr{otherwise}\\
\end{array}\qquad\,. \right .
\end{equation*}

Having a thorough analysis of~\eqr{spec3} as in previous section, we get the complete description of the spaces $E_{2}^{0,q}$ and $E_{2}^{1,q}$ which will give us the cohomology $H^{q}(\partial \rS_\Gamma, \tm_{\lambda})$ described as follows :
\begin{equation*}
H^q(\partial \rS_\Gamma, \tm_\lambda) =\left\{\begin{array}{cccc}
&  H^{1}_{!}(\rS_\Gamma^{\rm{M}_2}, \tm_{\lambda}) \,, & q=1 \\
&\\
& H^{1}_{!}(\rS_\Gamma^{\rm{M}_1}, \tm_{w_{3} \cdot \lambda}) \oplus H^{1}_{!}(\rS_\Gamma^{\rm{M}_2}, \tm_{w_2\cdot\lambda})  \,, & q=2 \\
& \\
&  H^{1}_{!}(\rS_\Gamma^{\rm{M}_1}, \tm_{w_5\cdot\lambda}) \,, & q=3 \\
&\\
& H^{1}_{!}(\rS_\Gamma^{\rm{M}_1}, \tm_{w_{7} \cdot \lambda})  \,, & q=4 \\
& \\
&  H^{1}_{!}(\rS_\Gamma^{\rm{M}_1}, \tm_{w_{9} \cdot \lambda}) \oplus  H^{1}_{!}(\rS_\Gamma^{\rm{M}_2}, \tm_{w_{8} \cdot \lambda}) \,, & q=5 \\
&\\
& H^{1}_{!}(\rS_\Gamma^{\rm{M}_2}, \tm_{w_{10}\cdot\lambda}) \,, & q=6 \\
&\\
& 0  \,,& \mr{otherwise}\\
\end{array}\qquad\,. \right .
\end{equation*}

We conclude the above discussion as follows :
\begin{equation*}
H^q(\partial \rS_\Gamma, \tm_\lambda) =\left\{\begin{array}{cccc}
&  S_{m_2 + 2}  \,, & q=1, 6 \\
& \\
&  S_{m_1 + 3m_2+ 5} \oplus S_{m_1 + m_2+3}\,, & q=2, 5 \\
&\\
&  S_{2m_1+3m_2 + 6}  \,, & q=3, 4 \\
& \\
& 0  \,,& \mr{otherwise}\\
\end{array}\qquad\,. \right .
\end{equation*}

\subsection{Case 9:   $m_1 $ odd, $m_2$ odd}
By checking the parity conditions for standard parabolics, following Lemmas ~\ref{parity0} and ~\ref{parity1}, we see that $\overline{\W}^{i} = \emptyset$ for $i=0,1,2$. This simply implies that
$$H^{q}(\partial \rS_\Gamma, \tm_\lambda) = 0 \,, \qquad \forall q \,.$$


 \subsection{A summary of the boundary cohomology of $G_2(\Z)$}
 We now end this section by summarizing the results obtained about the boundary cohomology above  in the form of following theorem which will be our base for further exploration on Eisenstein cohomology in Section~\ref{Eis}.

\begin{thm}\label{bdg2} The boundary cohomology of  the locally symmetric space $\mr{S}_{\G}$ of the arithmetic group $\G:=G_2(\Z)$ with respect to the coefficients in any highest  weight representation $\m_\lambda$, with $\lambda=m_1\lambda_1 + m_2 \lambda_2$, is described as follows.
\begin{enumerate}
\item Case 1 : $m_1 = 0 = m_2$.

\begin{equation*}
H^q(\partial \rS_\Gamma, \tm_\lambda) =\left\{\begin{array}{cccc}
&\mathbb{Q} \,, & q=0,7 \\
&\\
&0  \,, & \mr{otherwise}\\
\end{array}\qquad\,. \right .
\end{equation*}

\item Case 2 : $ m_1 = 0$, $m_2 (\neq 0)$ even.

\begin{equation*}
H^q(\partial \rS_\Gamma, \tm_\lambda) =\left\{\begin{array}{cccc}
& S_{m_2 + 2} \,, & q=1, 6 \\
&\\
& S_{3m_2 + 6} \oplus S_{2m_2 + 4}\,, & q=3, 4 \\
&\\
& 0\,, & \mr{otherwise}\\
\end{array}\qquad\,. \right .
\end{equation*}

\item Case 3 : $ m_1=0$ , $m_2$ odd.

\begin{equation*}
H^q(\partial \rS_\Gamma, \tm_\lambda) =\left\{\begin{array}{cccc}
& S_{3m_2 + 5} \oplus S_{m_2 + 3} \oplus \Q  \,, & q=2, 5 \\
&\\
& S_{2m_2 + 4} \,, & q=3, 4 \\
&\\
& 0  \,, & \mr{otherwise}\\
\end{array}\qquad\,. \right .
\end{equation*}

\item Case 4 : $m_1 (\neq  0)$ even, $m_2 = 0$.

\begin{equation*}
H^q(\partial \rS_\Gamma, \tm_\lambda) =\left\{\begin{array}{cccc}
&  S_{m_1 + 2}  \,, & q=1, 6 \\
& \\
&  S_{2m_1 + 6} \oplus S_{m_1 + 4}\,, & q=3, 4 \\
&\\
& 0  \,, & \mr{otherwise}\\
\end{array}\qquad\,. \right .
\end{equation*}

\item Case 5 : $ m_1 (\neq  0)$ even, $m_2 (\neq 0)$ even.

\begin{equation*}
H^q(\partial \rS_\Gamma, \tm_\lambda) =\left\{\begin{array}{cccc}
&  S_{m_1 + 2} \oplus S_{m_2 + 2} \oplus \Q \,, & q=1, 6 \\
& \\
&  S_{2m_1 + 3m_2+ 6} \oplus S_{m_1 + 2m_2+4}\,, & q=3, 4 \\
&\\
& 0  \,, & \mr{otherwise}\\
\end{array}\qquad\,. \right .
\end{equation*}

\item Case 6 : $ m_1 (\neq  0)$ even, $m_2$ odd.

\begin{equation*}
H^q(\partial \rS_\Gamma, \tm_\lambda) =\left\{\begin{array}{cccc}
&  S_{m_1 + 2} \,, & q=1, 6 \\
& \\
&  S_{m_1 + 3m_2+ 5} \oplus S_{m_1 + m_2+3}\,, & q=2, 5 \\
&\\
&  S_{m_1 + 2m_2+ 4} \,, & q=3, 4 \\
&\\
& 0  \,, & \mr{otherwise}\\
\end{array}\qquad\,. \right .
\end{equation*}

\item Case 7 : $ m_1$ odd, $m_2 = 0$.

\begin{equation*}
H^q(\partial \rS_\Gamma, \tm_\lambda) =\left\{\begin{array}{cccc}
& S_{m_1 + 5} \oplus S_{m_1 + 3} \oplus \Q  \,, & q=2, 5 \\
&\\
& S_{2m_1 + 6} \,, & q=3, 4 \\
&\\
& 0  \,, & \mr{otherwise}\\
\end{array}\qquad\,. \right .
\end{equation*}

\item Case 8 : $ m_1$ odd, $m_2 (\neq 0)$ even.
\begin{equation*}
H^q(\partial \rS_\Gamma, \tm_\lambda) =\left\{\begin{array}{cccc}
&  S_{m_2 + 2}  \,, & q=1, 6 \\
& \\
&  S_{m_1 + 3m_2+ 5} \oplus S_{m_1 + m_2+3}\,, & q=2, 5 \\
&\\
&  S_{2m_1+3m_2 + 6}  \,, & q=3, 4 \\
& \\
& 0  \,, & \mr{otherwise}\\
\end{array}\qquad\,. \right .
\end{equation*}

\item Case 9 : $ m_1$ odd, $m_2$ odd.

$$H^{q}(\partial \rS_\Gamma, \tm_\lambda) = 0 \,, \qquad \forall q \,.$$

\end{enumerate}
\end{thm}




\section{Eisenstein cohomology}\label{Eis}
In this section, by using the information obtained about boundary cohomology  of $\G:=G_{2}(\Z)$, we will determine the Eisenstein cohomology with coefficients in $\m_\lambda$. Let us recall that, at any degree $q$, the Eisenstein cohomology  $H^q_{Eis}(\rS_\Gamma, \tm_\lambda)$ is, by definition, the  the image of the restriction map  $r: H^q(\rS_\Gamma, \tm_\lambda) \longrightarrow H^q(\partial \rS_\Gamma, \tm_\lambda) $.

\subsection{Main result on the Eisenstein cohomology of $G_2(\Z)$}

The following is one of the main result of this article, that gives both the dimension of the Eisenstein cohomology together with its sources - the corresponding parabolic subgroups.

Indeed, it is clear from the definition that the Eisenstein cohomology $H^\bullet_{Eis}(\rS_\Gamma, \tm_\lambda)$ is defined over $\Q$ as $\tm_\lambda$ is defined over $\Q$. But in the theorem stated below, we will consider $H^\bullet_{Eis}(\rS_\Gamma, \tm_\lambda\otimes \C)$ instead. The reason is that, our method  yields a basis of $H^\bullet_{Eis}(\rS_\Gamma, \tm_\lambda\otimes\C)$, but since the method is transcendental, the basis we get is not necessarily defined over $\Q$.
Let $\Sigma_k$ be the canonical basis of normalized eigenfunctions of $S_k$. For $i=1,2$, we set $k_i(\lambda, w)=a_i(\lambda, w)+2$, where $a_i(\lambda, w)$ is the constant defined in \eqref{def-a-b}.   Then by the Eichler-Shimura isomorphism, for $i=1,2$, we have
\begin{align*}
  H^1_{!}(\rS^{\rM_i}_\Gamma, \tm_{w\cdot \lambda}\otimes \C)=\bigoplus_{\psi\in \Sigma_{k_i(\lambda, w)}}H^1_{!}(\rS^{\rM_i}_\Gamma, \tm_{w\cdot \lambda}\otimes \C)(\psi)\,,
\end{align*}
where the $\C$-vector spaces $H^1_{!}(\rS^{\rM_i}_\Gamma, \tm_{w\cdot \lambda}\otimes \C)(\psi)$ are of dimension $1$. Set
\begin{equation*}
  \cY_k=\{\psi\in \Sigma_k: L(1/2, \pi^\psi)\ne 0\} \quad \text{and} \quad \cZ_k=\{\psi \in \Sigma_k: L(1/2, Sym^3\pi^\psi)\ne 0\}.
\end{equation*}

\begin{thm}\label{Eiscoh}$\,\,\,$
\begin{enumerate}
\item Case 1 : $m_1 = 0=m_2$.
\begin{equation*}
H^{q}_{Eis}(\rS_\Gamma, \tm_\lambda\otimes \C) =\left\{\begin{array}{cccc}
&\mathbb{C} \quad \mr{for}\quad  q=0 \\
&\\
&0  \qquad \mr{otherwise}\\
\end{array}\qquad\,. \right .
\end{equation*}

\item Case 2 : $ m_1 = 0$, $m_2 (\neq 0)$ even.

\begin{equation*}
H^{q}_{Eis}(\rS_\Gamma, \tm_\lambda\otimes \C) =\left\{\begin{array}{cccc}
&  \bigoplus\limits_{\psi\in \cZ_{2m_2+4}} \C \psi\,, & q=3 \\
&\\
& S_{3m_2 + 6} \oplus \bigoplus\limits_{\psi\notin \cZ_{2m_2+4}}  \C \psi \,, & q=4 \\
& \\
&  S_{m_2 + 2} \,, & q=6 \\
&\\
& 0 \,, & \mr{otherwise}\\
\end{array}\qquad\,. \right .
\end{equation*}

\item Case 3 : $ m_1=0$ , $m_2$ odd.

\begin{equation*}
H^{q}_{Eis}(\rS_\Gamma, \tm_\lambda\otimes \C) =\left\{\begin{array}{cccc}
& \bigoplus\limits_{\psi\in \cZ_{2m_2+4}}  \C \psi\,, & q=3 \\
&\\
&\bigoplus\limits_{\psi\notin \cZ_{2m_2+4}}  \C \psi \,, & q=4 \\
& \\
&  S_{3m_2+5} \oplus S_{m_2+3} \oplus \C \,, & q=5 \\
&\\
& 0  \,, & \mr{otherwise}\\
\end{array}\qquad\,. \right .
\end{equation*}

\item Case 4 : $m_1 (\neq  0)$ even, $m_2 = 0$.

\begin{equation*}
H^{q}_{Eis}(\rS_\Gamma, \tm_\lambda\otimes \C) =\left\{\begin{array}{cccc}
&  \bigoplus\limits_{\psi\in \cY_{2m_2+6}}  \C \psi\,, & q=3 \\
&\\
&   \bigoplus\limits_{\psi\notin \cY_{2m_2+6}}  \C \psi \oplus S_{m_1+4} \,, & q=4 \\
& \\
&  S_{m_1 + 2} \,, & q=6 \\
&\\
& 0 \,, & \mr{otherwise}\\
\end{array}\qquad\,. \right .
\end{equation*}

\item Case 5 : $ m_1 (\neq  0)$ even, $m_2 (\neq 0)$ even.

\begin{equation*}
H^{q}_{Eis}(\rS_\Gamma, \tm_\lambda\otimes \C) =\left\{\begin{array}{cccc}
&  S_{2m_1 + 3m_2+ 6} \oplus S_{m_1 + 2m_2+4}\,, & q=4 \\
& \\
&  S_{m_1 + 2} \oplus S_{m_2 + 2} \oplus \C \,, & q=6 \\
&\\
& 0  \,, & \mr{otherwise}\\
\end{array}\qquad\,. \right .
\end{equation*}

\item Case 6 : $ m_1 (\neq  0)$ even, $m_2$ odd.

\begin{equation*}
H^{q}_{Eis}(\rS_\Gamma, \tm_\lambda\otimes \C) =\left\{\begin{array}{cccc}
&  S_{m_1 + 2m_2+ 4} \,, & q=4 \\
& \\
&  S_{m_1 + 3m_2+ 5} \oplus S_{m_1 + m_2+3}\,, & q=5 \\
&\\
&  S_{m_1 + 2} \,, & q=6 \\
&\\
& 0  \,, & \mr{otherwise}\\
\end{array}\qquad\,. \right .
\end{equation*}

\item Case 7 : $ m_1$ odd, $m_2 = 0$.

\begin{equation*}
H^{q}_{Eis}(\rS_\Gamma, \tm_\lambda\otimes \C) =\left\{\begin{array}{cccc}
&  \bigoplus\limits_{\psi\in \cY_{2m_2+6}}  \C \psi\,, & q=3 \\
&\\
&   \bigoplus\limits_{\psi\notin \cY_{2m_2+6}}  \C \psi \,, & q=4 \\
& \\
&  S_{m_1 + 5} \oplus S_{m_1 + 3} \oplus \C \,, & q=5\\
&\\
& 0 \,, & \mr{otherwise}\\
\end{array}\qquad\,. \right .
\end{equation*}

\item Case 8 : $ m_1$ odd, $m_2 (\neq 0)$ even.
\begin{equation*}
H^{q}_{Eis}(\rS_\Gamma, \tm_\lambda\otimes \C) =\left\{\begin{array}{cccc}
&  S_{2m_1+3m_2 + 6}  \,, & q=4 \\
& \\
&  S_{m_1 + 3m_2+ 5} \oplus S_{m_1 + m_2+3}\,, & q=5 \\
&\\
&  S_{m_2 + 2}  \,, & q=6 \\
& \\
& 0  \,, & \mr{otherwise}\\
\end{array}\qquad\,. \right .
\end{equation*}

\item Case 9 : $ m_1$ odd, $m_2$ odd.

$$H^{q}_{Eis}(\rS_\Gamma, \tm_\lambda) = 0 \,, \qquad \forall q \,.$$

\end{enumerate}
\end{thm}

Now, the proof of Theorem~\ref{Eiscoh} will occupy the rest of the paper, and will follow in several steps. The proof closely follow the strategy developed in~\cite{Harder2012} (see also~\cite{LiSch93}, \cite{Sch86} and~\cite{Sch90}).

\subsection{General Strategy.}
 We now briefly describe the strategy which will be detailed out in the rest  of this section. As aforementioned, our approach  relies on the fact that the Eisenstein cohomology spans a maximal isotropic subspace of the boundary cohomology with respect to the Poincar\'e dual pairing (see Theorem \ref{T-dual}). Indeed, certain cohomology classes are constructed using the theory of Eisenstein series, we are done if the cohomology classes thus constructed spans a maximal isotropic subspace.

More precisely, let $\omega$ be a  harmonic differential form that represents certain cohomology class in the boundary cohomology $H^{q}_!(\partial_i, \tm_\lambda) (i=0, 1, 2)$. By mimicking the construction of Eisenstein series, we get a family of differential forms $E(\omega, \theta)$ on $\rS_\Gamma$, where $\theta$ is a certain parameter, will be discussed in Section~\ref{sec:ecc}. If the differential forms $E(\omega, \theta)$ is holomorphic at certain $\theta_{\omega}$, we get a non-trivial harmonic form $E(\omega, \theta_\omega)$ that represents certain cohomology class in $H^{q}(\rS_\Gamma, \tm_\lambda)$ on $\rS_\Gamma$. By restricting the harmonic form back to the boundary, we get a non-trivial Eisenstein cohomology class in $H^{q}_{Eis}(\partial \rS_\Gamma, \tm_\lambda)$. This geometric formulation is closely related to the theory of $(\fg, K_\infty)$-cohomology and the classical Eisenstein series. In particular, the restriction of the cohomology class $E(\omega, \theta_\omega)$ to the boundary can be computed using the constant term of Eisenstein series.
On the other hand, if the the differential form $E(\omega, \theta)$ has a simple pole at $\theta_{\omega}$ (or along some hyperplane that contains $\theta_{\omega}$), by taking residue, we still get a differential form $E'(\omega, \theta_\omega)$, whose restriction to the boundary also gives certain Eisenstein cohomology class.

The study of Eisenstein cohomology is basically divided into two parts. In the first part, we study the Eisenstein cohomology classes that comes from maximal boundary components, that is, those constructed from cohomology class in $H^{\bullet}_!(\partial_i, \tm_\lambda)$ with $i=1, 2$.
In the second part, we study the Eisenstein cohomology classes that comes from the minimal boundary component, that is, those constructed from cohomology class in $H^{\bullet}_!(\partial_0, \tm_\lambda)$.

\subsection{Poincar\'e duality.}\label{sec:poincare}
For simplicity, we write
\begin{equation*}
  H^q_!(\partial \rS_\Gamma, \tm_\lambda):= H^q_!(\partial_1, \tm_\lambda)\oplus H^q_!(\partial_2, \tm_\lambda) \subset H^q(\partial \rS_\Gamma, \tm_\lambda),
\end{equation*}
and
\begin{equation*}
  H^q_{!, Eis}(\partial \rS_\Gamma, \tm_\lambda):= H^q_{Eis}(\rS_\Gamma, \tm_\lambda)\cap H^q_{!}(\partial \rS_\Gamma, \tm_\lambda).
\end{equation*}

We shall need the following theorem from~\cite{HR2020}.
\begin{thm}\label{T-dual}
 Under the Poincar\'e dual pairing $\langle\cdot, \cdot \rangle$
$$H^q(\partial \rS_\Gamma, \tm_\lambda)\times H^{7-q}(\partial \rS_\Gamma, \tm_\lambda)\rightarrow \Q, \quad H_!^q(\partial \rS_\Gamma, \tm_\lambda)\times H_!^{7-q}(\partial \rS_\Gamma, \tm_\lambda)\rightarrow \Q,$$
  we have
    $$H^q_{Eis}(\partial \rS_\Gamma, \tm_\lambda)=H^{7-q}_{Eis}(\partial \rS_\Gamma, \tm_\lambda)^\perp, \quad H^q_{!,Eis}(\partial \rS_\Gamma, \tm_\lambda)=H^{7-q}_{!,Eis}(\partial \rS_\Gamma, \tm_\lambda)^\perp .$$
    In particular, the Eisenstein cohomology is a maximal isotropic subspace of the boundary cohomology under the Poincar\'e duality.
\end{thm}

 Let $\mathbb{A}$ (resp. $\mathbb{A}_f$) be the ring of adeles (resp. finite adeles) of $\Q$ and $K_f=\prod_{p} G_2(\Z_p)$. It is clear that the Poincar\'e dual pairings are Hecke equivariant,  hence Theorem~\ref{T-dual} can be further refined by considering the Hecke action.
 Now let $i=1,2$,  the inner cohomology $H^\bullet_!(\partial_i, \tm_\lambda)$, considered as a $G_2(\A_f)$ module, can be decomposed as
\begin{align*}
  H^\bullet_!(\partial_i, \tm_\lambda) &= \bigoplus_{w\in \W^{\rP_i}} H^{1+\ell(w)}(S^{\rM_i}_\Gamma, \tm_{w\cdot \lambda}) \\
   &=\bigoplus_{w\in \W^{\rP_i}} \bigoplus_{\pi=\pi_{\infty}\otimes \pi_f} m_0(\pi) H^{1+\ell(w)} (\fm_i, K^{\rM_i}_\infty, \pi_\infty\otimes M_{w\cdot \lambda})(\pi_f).
\end{align*}
Here $\pi$ denotes a cuspidal automorphic representation of $\rM_i(\A)$ with unramified $\pi_f$ and $m_0(\pi)$ denotes the multiplicity of $\pi$. A cohomology class in $H^\bullet_!(\partial_i, \tm_\lambda)$ is said to be of type $(\pi, w)$ if it comes from the summand $H^{1+\ell(w)} (\fm_i, K^{\rM_i}, \pi_\infty\otimes M_{w\cdot \lambda})(\pi_f)$ in the above decomposition.

Now let $\beta_1, \beta_2\in H^\bullet_!(\partial_i, \tm_\lambda)$ be cohomology classes of type $(\pi_1, w_1)$ and $(\pi_2, w_2)$ respectively. Recall that the Poincare dual pairing is also $G_2(\A_f)$ equivariant, hence $\langle\beta_1, \beta_2\rangle=0$ if $\pi_{1, f}\ne \pi_{2,f}$, which implies $\pi_1 \ne \pi_2$ be strong multiplicity one. On the other hand, for dimension reasons, $\langle\beta_1, \beta_2\rangle=0$ only when $\ell(w_1)+\ell(w_2)=5$, which is equivalent to saying that $w_1$ is mapped to $w_2$ under the involution introduced in Lemma \ref{lem-involution}. In conclusion, we get
\begin{lema}\label{L-hec-dual}
Let $i=1, 2$ and $\beta_1, \beta_2\in H^\bullet_!(\partial_i, \tm_\lambda)$ be cohomology classes of type $(\pi_1, w_1)$ and $(\pi_2, w_2)$ respectively. Then $\langle\beta_1, \beta_2\rangle$ is nonzero only if $\pi_1=\pi_2$ and $w_1=w_2'$.
\end{lema}

Let $\W^{\rP_i}_>=\{w\in \W^{\rP_i}: \ell(w)\ge \ell(w')\}$.
In view of this lemma, we may regroup $H_{!}^\bullet(\partial_i, \tm_\lambda)$ using parameters $w\in \W^{\rP_i}_>$ as follows:
\[ H^\bullet_{!}(\partial_i, \tm_\lambda)=\bigoplus_{w\in \W^{\rP_i}_>}\bigoplus_{\psi\in \Sigma_{k_i(\lambda, w)}}
H^\bullet_{!}(\partial_i, \tm_\lambda)(\pi^\psi_f, w),\]
with
\begin{equation*}
  H^\bullet_{!}(\partial_i, \tm_\lambda)(\pi^\psi_f, w) :=H^{1+\ell(w')}_{!}(\rS^{\rM_i}_\Gamma, \tm_{w'\cdot \lambda})(\pi^\psi_f) \oplus H^{1+\ell(w)}_{!}(\rS^{\rM_i}_\Gamma, \tm_{w\cdot \lambda})(\pi^\psi_f),
\end{equation*}
where $\pi^\psi$ denotes the automorphic representation associated to the Hecke eigenform $\psi$.
By the multiplicity one theorem, $\dim H^{1+\ell(w)}_{!}(\rS^\rM, \tm_{w\cdot \lambda})(\pi^\psi_f)=1$ for any $(\pi^\psi, w)$. Hence, by combining Theorem \ref{T-dual} and Lemma \ref{L-hec-dual}, we get
\begin{prop}\label{prop-main}
The Eisenstein cohomology $H^\bullet_{!,Eis}(\partial \rS_\Gamma, \tm_\lambda)$ decomposes as
\begin{equation*}
  H^\bullet_{!,Eis}(\partial \rS_\Gamma, \tm_\lambda)=H^\bullet_{!,Eis}(\partial_1, \tm_\lambda)\oplus H^\bullet_{!,Eis}(\partial_2, \tm_\lambda),
\end{equation*}
where
\begin{equation*}
  H^\bullet_{!,Eis}(\partial_i, \tm_\lambda)=\bigoplus_{w\in \W^{\rP_i}_>}\bigoplus_{\psi\in \Sigma_{k_i(\lambda, w)}} H_{!,Eis}(\partial_i, \tm_\lambda)(\pi^\psi, w),
\end{equation*}
with $H^\bullet_{!,Eis}(\partial_i, \tm_\lambda)(\pi^\psi, w)$ equals either $H^{1+\ell(w')}_{!}(\rS^{\rM_i}_\Gamma, \tm_{w'\cdot \lambda})(\pi^\psi_f) $ or $H^{1+\ell(w)}_{!}(\rS^{\rM_i}_\Gamma, \tm_{w\cdot \lambda})(\pi^\psi_f)$.
\end{prop}

\subsection{Eisenstein forms.}\label{sec:ecc}  Let  $\omega\in \Omega^*(\partial_i, \tm_{\lambda}) (i=0,1,2)$ be a differential form on the boundary component $\partial_i$ and $\widetilde{\omega}\in \Omega^*(\Gamma_{\rP_i}\backslash \rS, \tm_\lambda)$ be the pull-back of $\omega$ along the projection of $\Gamma_\rP\backslash \rS=\partial_i \times A_{\rP_i}$ to the first factor. For any $\theta\in \fb_i:=\fa_{\rP_i}^*\otimes \C$, set
$$\omega_{\theta}=\widetilde{\omega}\times a^{\theta+\rho_{i}}\in \Omega^*(\Gamma_{\rP_i}\backslash \rS, \tm_\lambda),$$
where $\rho_{i}=\rho\vert_{\fb_i}$
Then we define the corresponding Eisenstein form as
\begin{equation*}
  E(\omega, \theta)=\sum_{\gamma\in \Gamma/\Gamma_{\rP_i}} \omega_{\theta}\circ \gamma \in \Omega^*(\rS_\Gamma, \tm_\lambda).
\end{equation*}
It is well-known that the series $E(\omega, \theta)$ converges in certain region and admits an analytic continuation to a meromorphic function on $\fb_i$.

\subsection{Constant terms and intertwining operators.}

To proceed, we consider the representation theoretic reformulation of the Eisenstein forms. The complex of smooth forms $\Omega^*(\rS_\Gamma, \tm_{\lambda})$ can be computed as
\begin{align*}
 \Omega^*(\rS_\Gamma, \tm_{\lambda}) &= C^*(\fg, K_\infty; C^{\infty}(\Gamma\backslash G_2(\R))\otimes \m_\lambda) \\
   &= C^*(\fg, K_\infty; C^{\infty}(G_2(\Q)\backslash G_2(\A))\otimes \m_\lambda)^{K_f}.
\end{align*}
Consequently, we have
\begin{align*}
H^\bullet(\rS_\Gamma, \tm_\lambda)&=H^\bullet(\fg, K_\infty; C^{\infty}(G_2(\Q)\backslash G_2(\A))\otimes \m_\lambda)^{K_f}
\end{align*}
Now let $\rP_i$ be a standard parabolic, $\pi=\pi_{\infty}\otimes \pi_f$ be a cuspidal representation of $\rM_i(\A)$ and let $\theta\in \fb_i$ be a parameter. For $\psi_\theta\in V(\theta, \pi):=\Ind_{\rP_i}^{G_2}\pi\otimes \C^{\theta+\rho_i}$, where $\C^{\theta+\rho_i}$ denotes the one-dimensional representation of $\rP_i(\A)$ given by the character $\theta+\rho_i$.
Define the Eisenstein series as
\begin{equation*}
  E_{\rP_i}(\theta,\pi, \psi_\theta)(g)=\sum_{\gamma\in \rP_i(\Q)\backslash G_2(\Q)} \psi_\theta(\gamma g).
\end{equation*}
For $\theta$ from certain region, the Eisenstein series defined above converges absolutely, hence defines an intertwining operator
\begin{equation*}
  \Eis_{\theta}: V(\theta, \pi) \rightarrow C^\infty(G_2(\Q)\backslash G_2(\A)),
\end{equation*}
in this region.
The Eisenstein series  has a meromorphic continuation to $\fb_i$, hence defines a meromorphic continuation of the corresponding intertwining operator. When the Eisenstein series has a simple pole at $\theta$,  by taking the derivative, we get an intertwining operator
\begin{equation*}
  \Eis_{\theta}': V(\theta, \pi) \rightarrow C^\infty(G_2(\Q)\backslash G_2(\A)).
\end{equation*}
It is clear that, for a closed form $\omega \in \Omega^*(\partial_i, \tm_{\lambda})$ that represents certain cohomology class of type $(\pi, w)$,  the intertwining operator $\Eis$ is holomorphic (resp. have a simple pole) at $\theta$ if and only if the corresponding Eisenstein form $E(\omega, \theta)$ is holomorphic (resp. have a simple pole) at $\theta$.

To determine whether the Eisenstein series is holomorphic or not, it suffices to look at the constant terms.
First, we consider the case when the parabolic subgroup is maximal.  Let $\rQ=\rM_\rQ \rN_\rQ$ be another standard parabolic subgroup. Then the constant term along $\rQ$ is defined as
\begin{equation*}
  E_\rQ(\theta,\pi, \psi_\theta)(g)=\int_{\rN_\rQ(\Q)\backslash \rN_\rQ(\A)} E(\theta,\pi, \psi_\theta)(ng)dn.
\end{equation*}
Since the maximal parabolic subgroups of $G_2$ are self-conjugate, the constant term is non-trivial only when $\rQ=\rP_i$, where we have
\begin{equation*}
  E_{\rP_i}(\theta, \pi, \psi_\theta)(g)=\psi_\theta(g)+M(\theta, \pi, w_{\rP_i})\psi_\theta(g),
\end{equation*}
where $w_{\rP_i}$ is the longest element in $\W^{\rP_i}$ and $M(\theta, \pi, w_{\rP_i})$ is a global intertwining operator from $V( \theta, \pi)$ to $V(-\theta, \pi)$. The global intertwining operator $M(\theta, \pi, w_{\rP_i})$ is a product of local intertwining operator
$$M(\theta, \pi, w_{\rP_i})=A(\theta, \pi_{\infty}, w_{\rP_i})\otimes A(\theta, \pi_f, w_{\rP_i}), \text{ where } A(\theta, \pi_f, w_{\rP_i})=\otimes_p A(\theta, \pi_p, w_{\rP_i}).$$
Note that if $\pi_p$ is unramified, the induced representation $V(\theta, \pi_p)$ is also unramified.
Works of Langlands and Gindikin-Karpelevich, see, for example \cite{cogdell2004lectures}, give a description of the local intertwining operator on $G_2(\Z_p)$-invariant vectors. As a consequence, we have the following description of the global intertwining operator.
\begin{lema}$\,\,\,$
\begin{itemize}
  \item[(1)] Let $i=1$ and $\theta=z\gamma_2\in \fb_1$. Then,
  \begin{equation*}
    M(\theta, \pi, w_{\rP_1})= c_1(\theta, \pi)A(\theta, \pi_{\infty}, w_{\rP_1})\otimes A'(\theta, \pi_f, w_{\rP_1})
 \end{equation*}
  where $A'(\theta, \pi_f, w_{\rP_i})$ is the intertwining operator from $V(\theta, \pi_f):=\otimes_p V(\theta, \pi_p)$ to $V(-\theta, \pi_f)$ that sends a normalized $K$-invariant vector in $V(\theta, \pi_f)$ to a normalized $K$-invariant vector in $V(-\theta, \pi_f)$ and
    \begin{equation*}
    c_1(\theta, \pi)= \frac{L(z, \pi)}{L(z+1, \pi)} \frac{\zeta(2z)}{\zeta(2z+1)} \frac{L(3z, \pi)}{L(3z+1, \pi)}.
\end{equation*}

  \item[(2)] Let $i=2$ and $\theta=z\gamma_1 \in \fb_2$. Then,
  \begin{equation*}
    M(\theta, \pi, w_{\rP_2})= c_2(\theta, \pi)A(\theta, \pi_{\infty}, w_{\rP_2})\otimes A'(\theta, \pi_f, w_{\rP_2})
  \end{equation*}
  where $A'(\theta, \pi_f, w_{\rP_2})$ is the intertwining operator from $V(\theta, \pi_f)$ to $V(-\theta, \pi_f)$ that sends a normalized $K$-invariant vector in $V(\theta, \pi_f)$ to a normalized $K$-invariant vector in $V(-\theta, \pi_f)$ and
    \begin{equation*}
    c_2(\theta, \pi)= \frac{L(z, Sym^3\pi)}{L(z+1, Sym^3\pi)}  \frac{\zeta(2z)}{\zeta(2z+1)}.
\end{equation*}
\end{itemize}
\end{lema}
\begin{proof}
  Recall that the factor of the intertwining operator for $G_2$ is given by the action of the adjoint action of the L-group of the Levi component of the uniponent radical. For  maximal parabolic subgroup, the Levi is isomorphic to $\rGL_2$. Hence the L-group is isomorphic to $\rGL_2(\C)$.
  Let $V\cong \C^2$ be the standard representation of $\rGL_2(\C)$ of dimension $2$. As determined in \cite{cogdell2004lectures}, the adjoint action of $\rGL_2(\C)$ on ${}^L\n_1$ and the adjoint action of $\rGL_2(\C)$ on ${}^L\n_2$ decompose as
  \begin{equation*}
    {}^L\n_1=V\otimes \wedge^2V \oplus V \oplus \wedge^2V, \quad {}^L\n_2=Sym^3V\otimes(\wedge^2V)^{-1}\oplus \wedge^2V.
  \end{equation*}
Since $\pi$ is unramified,  the corresponding automorphic representation of $\wedge^2 \pi_f$ is trivial. Hence we have the factor
  \begin{equation*}
    c_1(z, \pi)= \frac{L(z, \pi)}{L(z+1, \pi)} \frac{\zeta(2z)}{\zeta(2z+1)} \frac{L(3z, \pi)}{L(3z+1, \pi)}.
\end{equation*}
    and the factor for $\rP_2$ is given by
  \begin{equation*}
    c_2(z, \pi)= \frac{L(z, Sym^3\pi)}{L(z+1, Sym^3\pi)}  \frac{\zeta(2z)}{\zeta(2z+1)}.
\end{equation*}
\end{proof}

Now let $i=0$. Here we consider the special case when $\pi=\C$ is the trivial representation. For $\theta=z_1\gamma_1+z_2\gamma_2\in \fb_0$, set $V(\theta)=V(\theta, \C)$. Then the constant term can be computed as
  \begin{equation*}
  E_{\rP_0}(\theta, \psi_\theta)(g)=\sum_{w\in W} M(\theta, w)(\psi_\theta),
\end{equation*}
where $M(\theta, w)$ denotes a global intertwining functor from $V(\theta)$ to $V(w\cdot \theta)$. Again, the global intertwining operator is a product of local intertwining operators
$$M(\theta,\pi, w)=A(\theta,\pi_\infty,  w)\otimes A(\theta, \pi_f,  w), \text{ where } A(\theta, \pi_f,  w)=\otimes_p A(\theta, \pi_p, w).$$
\begin{lema}\label{lem-minimal}
  Let $i=0$ and $\theta=z_1\gamma_1+z_2\gamma_2$. Then
  \begin{equation*}
    M(\theta, \pi, w)= c_0(\theta, w)A(\theta, \pi_\infty, w)\otimes A'(\theta, \pi_f, w)
  \end{equation*}
  where $A'(\theta, \pi_f,  w)$ is the intertwining operator from $V(\theta, \pi_f):=\otimes_p V(\theta, \pi_p)$ to $V(w\cdot\theta, \pi_f)$ that sends a normalized $K$-invariant vector in $V(\theta, \pi_f)$ to a normalized $K$-invariant vector in $V(w\cdot\theta, \pi_f)$ and
\begin{equation*}
  c_0(\theta, w)= \prod_{\substack{ \alpha\in \Phi^+ \\ w^{-1}\alpha\in -\Phi^+ }} \frac{\zeta(\langle \alpha , \gamma_1 \rangle (z_1+1)+ \langle \alpha , \gamma_2 \rangle (z_2+1)-1)}{\zeta(\langle \alpha , \gamma_1 \rangle (z_1+1)+ \langle \alpha , \gamma_2 \rangle (z_2+1))}.
\end{equation*}
\end{lema}

\begin{proof}
This follows from direct computation, for a quick reference see~\cite[p. 159]{Harder2012} and ~\cite[Section 1.2.4]{Harder90} for the details.
\end{proof}

\subsection{The inner part of the Eisenstein cohomology.} Now we are ready to determine the space $H^\bullet_{!, Eis}(\partial \rS_\Gamma, \tm_\lambda)$. We begin with the following lemma.
\begin{lema}\label{l-poles}
Let $\lambda=m_1\gamma_1+m_2\gamma_2$ and set $\theta_{\lambda, w}^i:=-w(\lambda+\rho)\vert_{\fb_i} (i=1,2)$.
  \begin{itemize}
    \item[(1)]  The constant term $c_1(\theta, \pi)$ has a simple pole at $\theta_{\lambda, w}^1$ if  $w=w_7$, $m_2=0$ and $L(1/2, \pi)$ is non-zero and is holomorphic at $\theta_{\lambda, w}^1$ otherwise.
    \item[(2)]  The constant term $c_2(\theta, \pi)$ has a simple pole at $\theta_{\lambda, w}^2$ if $w=w_6$, $m_1=0$ and $L(1/2, Sym^3\pi)$ is non-zero and is holomorphic at $\theta_{\lambda, w}^1$ otherwise.
  \end{itemize}
\end{lema}

\begin{proof} As $\rho=\gamma^{M_1}+3\kappa^{M_1}$, we have
$$-w(\lambda+\rho)\vert_{\fb_1}=-\frac{1}{2}(t_i(w, \lambda)+3)\gamma_2.$$
Note that,  for $\theta=z\gamma_2$,
\begin{equation*}
    c_1(\theta, \pi)= \frac{L(z, \pi)}{L(z+1, \pi)} \frac{\zeta(2z)}{\zeta(2z+1)} \frac{L(3z, \pi)}{L(3z+1, \pi)}.
\end{equation*}
It is well-known that $L(z, \pi)$ appearing here are holomorphic and non-zero at $z$ when $\Re z>1$. Hence in view of Section 2.5, for $c_1(\theta, \pi)$ to have pole, it is necessary to have $w=w_7$, $m_2=0$. The possible pole comes from the simple pole of the zeta function at $2z=1$. But the simple pole may be canceled by a possible zero of $L(z, \pi)$ at $z=1/2$. This shows the (1).

As $\rho=\gamma^{M_2}+5\kappa^{M_2}$, we have
$-w(\lambda+\rho)\vert_{\fb_2}=-\frac{1}{2}(t_i(w, \lambda)+5)\gamma_2.$
Note that for $\theta=z\gamma_2$,
\begin{equation*}
    c_2(\theta, \pi)=\frac{L(z, Sym^3\pi)}{L(z+1, Sym^3\pi)}  \frac{\zeta(2z)}{\zeta(2z+1)}.
  \end{equation*}
Since the automorphic representations $\pi$ considered here are all unramified, the corresponding central character is trivial, hence $\pi$ is not monomial. Then, according to \cite{KS99}, the $L$-function $L(z, Sym^3\pi)$ is entire. Hence, in view of Section 2.5, for $c_2(\theta, \pi)$ to have pole it is necessary to have $w=w_6$, $m_1=0$. The possible pole comes from the simple pole of the zeta function at $2z=1$. But the simple pole may be canceled by a possible zero of $L(z, Sym^3\pi)$ at $z=1/2$. This shows the (2).
\end{proof}

We shall need the following theorem.
\begin{thm}\label{thm-maximal}
Let $i=1, 2$, $\pi$ be a cuspidal automorphic representation of $\rM_i(\A)$ and $w\in \W^{\rP_i}_>$. Let $\beta\in H^{1+\ell(w)}_!(\partial_i, \tm_\lambda)$ be a cohomology class of type $(\pi, w)$ and $\omega \in \Omega^*(\partial_i, \tm_{\lambda})$ be a closed harmonic form the represents $\beta$.
\begin{itemize}
  \item[(1)] If the Eisenstein series $E(\omega, \theta)$ is holomorphic at $\theta_{\lambda, w}^i$, then $E(\omega, \theta_{\lambda, w}^i)\in \Omega^*(\rS_\Gamma, \tm_\lambda)$ is a closed form such that the restriction of its cohomology class to the boundary $r([E(\omega, \theta_{\lambda, w}^i)])\in H^{1+\ell(w)}_!(\partial_i, \tm_\lambda)$ is non-trivial and of type $(\pi, w)$.
  \item[(2)] If the Eisenstein series $E(\omega, \theta)$ has a simple pole at $\theta_{\lambda, w}^i$, then the residue $E'(\omega, \theta_{\lambda, w}^i)\in \Omega^*(\rS_\Gamma, \tm_\lambda)$ is a closed form such that its restriction to the boundary $r([E'(\omega, \theta_{\lambda, w}^i)])\in H^{1+\ell(w')}_!(\partial_i, \tm_\lambda)$ is non-trivial and of type $(\pi, w')$.
\end{itemize}
\end{thm}

\begin{proof}
  When the group $G$ is of $\Q$-rank $1$, the corresponding statement are proved in~\cite{Harder75}. The same proof works when the cohomology classes come from the maximal boundary, hence works for this theorem as well. See also \cite{Sch86}.
\end{proof}

The subspace of $H^{\bullet}_!(\partial_i, \tm_\lambda)$ spanned by the cohomology classes of form $r([E(\omega,\theta_{\lambda, w}^i)])$ appearing in (1) (resp. $r([E'(\omega, \theta_{\lambda, w}^i)])$ appearing in (2)) is denoted as $H^{\bullet}_{reg}(\partial_i, \tm_\lambda)$ (resp, $H^{\bullet}_{res}(\partial_i, \tm_\lambda)$).
Then we have the natural inculsion
\begin{equation}\label{e-inc-1}
  \bigoplus_{i=1,2} \left(H^{\bullet}_{reg}(\partial_i, \tm_\lambda) \oplus   H^{\bullet}_{res}(\partial_i, \tm_\lambda)\right)\subset H^\bullet_{!, Eis}(\partial \rS_\Gamma, \tm_\lambda).
\end{equation}
By combining Proposition \ref{prop-main}, Lemma \ref{l-poles} and Theorem \ref{thm-maximal}, the above inclusion \eqref{e-inc-1} is indeed an equality.
Moreover,  we conclude that
\begin{prop}\label{l-last2}
Let $\lambda=m_1\gamma_1+m_2 \gamma_2$, $w\in \W^{\rP_i}_> (i=1,2)$ and $\psi\in \Sigma_{k_i(\lambda, w)}$. Then
  \begin{itemize}
    \item[(1)]
    \begin{equation*}
H^\bullet_{!,Eis}(\partial_i, \tm_\lambda)(\pi^\psi_f, w) =\left\{\begin{array}{cccc}
&H^{1+\ell(w')}_{!}(\rS^{\rM_1}_\Gamma, \tm_{w'\cdot \lambda})(\pi^\psi_f) \quad \mr{if}\  w=w_7, m_2=0, \mr{ and }\ L(1/2, \pi)\ne 0,  \\
&\\
&H^{1+\ell(w)}_{!}(\rS^{\rM_1}_\Gamma, \tm_{w\cdot \lambda})(\pi^\psi_f)  \quad \mr{otherwise}.\\
\end{array}\right .
\end{equation*}

\item[(2)]
    \begin{equation*}
H^\bullet_{!,Eis}(\partial_i, \tm_\lambda)(\pi^\psi_f, w) =\left\{\begin{array}{cccc}
&H^{1+\ell(w')}_{!}(\rS^{\rM_2}_\Gamma, \tm_{w'\cdot \lambda})(\pi^\psi_f) \quad \mr{if}\  w=w_6,  m_1=0, \mr{ and }\ L(1/2, Sym^3\pi)\ne 0, \\
&\\
&H^{1+\ell(w)}_{!}(\rS^{\rM_2}_\Gamma, \tm_{w\cdot \lambda})(\pi^\psi_f)  \quad \mr{otherwise}.\\
\end{array}\right .
\end{equation*}
  \end{itemize}
\end{prop}
Note that as the inculsion \eqref{e-inc-1} is an equality, in case $H^\bullet(\partial \rS_\Gamma, \tm_\lambda)=H^\bullet_!(\partial \rS_\Gamma, \tm_\lambda)$, we have already determined $H^\bullet_{Eis}(\rS_\Gamma, \tm_\lambda)$. Hence we are left to treat the cases 1, 3, 5 and 7 of Theorem~\ref{bdg2}.

\subsection{The boundary part of the Eisenstein cohomology.} In this section, we determine the Eisenstein cohomology classes that come from the minimal boundary $\partial_0$. As a consequence, we determine $H^\bullet_{Eis}(\rS_\Gamma, \tm_\lambda)$ for all the cases left. Throughout this subsection, we assume that $H^\bullet(\partial \rS_\Gamma, \tm_\lambda)\ne H^\bullet_!(\partial \rS_\Gamma, \tm_\lambda)$, or equivalently,  we are considering the cases 1, 3, 5 and 7 of Theorem~\ref{bdg2}.

Let $\beta$ be a cohomology class in $H^6(\widetilde{\partial_0}, \tm_\lambda)$\footnote{Here $\widetilde{\partial_0}$ denotes the cover of $\partial_0$, which is easily seen to be isomorphic to the unipotent radical $\rN$.},  and $\omega \in \Omega^6(\widetilde{\partial_0}, \tm_{\lambda})$ be a closed harmonic form that represents $\beta$. Recall that, as a $\rT$-module, we have
\[H^6(\widetilde{\partial_0}, \tm_\lambda)=H^0(\widetilde{\rS_\Gamma^{\rM_0}}, H^6(\rN, \tm_\lambda))=\C^{-\lambda-2\rho}.\]
The overall idea for construction of Eisenstein cohomology classes is the same as before. If the Eisenstein form $E(\omega, \theta)$ is holomorphic at $\theta_{\lambda}:=-\lambda-\rho$, then $E(\omega, \theta_\lambda)\in \Omega^6(\rS_\Gamma, \tm_\lambda)$ is a closed form such that the restriction of its cohomology class to the boundary is non-trivial, see~\cite[Theorem 7.2]{Sch94}. Otherwise, we need to take residues of the Eisenstein form and compute their restriction to the boundary using the constant term. As before, we denote the subspace of $H^{\bullet}(\partial \rS_\Gamma, \tm_\lambda)$ spanned the Eisenstein cohomology classes that comes from restriction of Eisenstein form (resp. residues of the Eisenstein form) as $H^{\bullet}_{B, reg}(\partial \rS_\Gamma, \tm_\lambda)$ (resp, $H^{\bullet}_{B, res}(\partial \rS_\Gamma, \tm_\lambda)$).

 For simplicity, set
\[H_{B, Eis}^{\bullet}(\partial \rS_\Gamma, \tm_\lambda)= H^{\bullet}_{B, reg}(\partial \rS_\Gamma, \tm_\lambda)\oplus H^{\bullet}_{B, res}(\partial \rS_\Gamma, \tm_\lambda).\]
Then we have the natrual inclusion
\begin{equation}\label{e-inc-2}
  H_{B, Eis}^\bullet(\partial \rS_\Gamma, \tm_\lambda) \oplus H^\bullet_{!, Eis}(\partial \rS_\Gamma, \tm_\lambda) \subset H^\bullet_{Eis}(\partial \rS_\Gamma, \tm_\lambda).
\end{equation}
According to Theorem \ref{T-dual}, the Eisenstein cohomology $H^\bullet_{Eis}(\partial \rS_\Gamma, \tm_\lambda)$ is a maximal isotropic subspace of boundary cohomology under the Poincar\'e duality. In particular, \[\dim H^\bullet_{Eis}(\partial \rS_\Gamma, \tm_\lambda)=\frac{1}{2} \dim H^\bullet(\partial \rS_\Gamma, \tm_\lambda).\]
On the other hand, according to Theorem \ref{bdg2}, we always have
\[\dim H^\bullet(\partial \rS_\Gamma, \tm_\lambda)- \dim H^\bullet_!(\partial \rS_\Gamma, \tm_\lambda)= 2,\]
for the cases studied in this subsection. Hence to determine $H^\bullet_{Eis}(\partial \rS_\Gamma, \tm_\lambda)$, it suffices to fill $H^\bullet_{B, Eis}(\partial \rS_\Gamma, \tm_\lambda)$ with one non-trivial cohomology class.

The following proposition is the main result of this subsection
\begin{prop}\label{l-last1}
Let notations be as in Theorem \ref{bdg2}.
\begin{itemize}
\item[(1)] In case 1, we have $H_{Eis}^{\bullet}(\partial \rS_\Gamma, \tm_\lambda)=H_{B, res}^{0}(\partial \rS_\Gamma, \tm_\lambda)=\Q$.
  \item[(2)] In case 3 and 7, we have $H_{B, Eis}^\bullet(\partial \rS_\Gamma, \tm_\lambda)=H^{5}_{B, res}(\partial \rS_\Gamma, \tm_\lambda)=\Q$.
  \item[(3)] In case 5, we have $H_{B, Eis}^\bullet(\partial \rS_\Gamma, \tm_\lambda)=H^{6}_{B, reg}(\partial \rS_\Gamma, \tm_\lambda)=\Q$.
\end{itemize}
\end{prop}

\begin{proof}
The proof will be based on the computation of $(\fg, K_\infty)$-cohomology of certain induced modules.
To begin, let us start with the proof of part (2). Without loss of generality, we may assume that $\lambda=m_2 \gamma_2$ for some $m_2>0$.

Recall that
\begin{equation*}
  c_0(\theta, w)= \prod_{\substack{ \alpha\in \Phi^+ \\ w^{-1}\alpha\in -\Phi^+ }} \frac{\zeta(\langle \alpha , \gamma_1 \rangle (z_1+1)+ \langle \alpha , \gamma_2 \rangle (z_2+1)-1)}{\zeta(\langle \alpha , \gamma_1 \rangle (z_1+1)+ \langle \alpha , \gamma_2 \rangle (z_2+1))}.
\end{equation*}
Hence,  in case of $w=s_1$ the constant term $c_0(\theta, w)$ has a simple pole along $z_1=0$.
Hence, the corresponding Eisenstein series has a simple pole along the line $z_1=0$. By taking the residue along the line $z_1=0$, we get an intertwining operator
\[\Eis': \quad \Ind_{\rP_0}^{G_2} \C^{-\lambda-\rho} \rightarrow C^{\infty}(G_2(\Q)\backslash G_2(\A)),\]
which factors through the Langlands quotient $J_\lambda$. More precisely, we have the following diagram.
\[\begin{tikzcd}
\Ind_{\rP_0}^{G_2} \C^{-\lambda-\rho}
\arrow[bend left]{drr}{\Eis'}
\arrow[bend right]{ddr}{\rho}
\arrow{dr}{A(-\lambda-\rho, \pi_\infty, s_1)} & & \\
& \Ind_{\rP_0}^{G_2} \C^{-w_1(\lambda+\rho)}
& C^{\infty}(G_2(\Q)\backslash G_2(\A)) \\
& J_\lambda \arrow[hookrightarrow]{u} \arrow[hookrightarrow]{ur}{\phi}
\end{tikzcd}\]
Note that,
\[J_\lambda=\Ind_{\rP_1}^{G_2} \C^{-w_1(\lambda+\rho)},\]
 and the intertwining operator $\rho$ is simply the induction of the intertwining operator
 \[\Ind_{\rB}^{\rM_1} \C^{-\lambda-\rho} \rightarrow  \C^{-w_1(\lambda+\rho)},\]
 where $\rB$ denotes the corresponding Borel subgroup of $\rM_1$.
 Namely, $\rho$ is the map
 \[\Ind_{\rP_1}^{G_2} \Ind_{\rB}^{\rM_1}\C^{-\lambda-\rho} \rightarrow \Ind_{\rP_1}^{G_2}\C^{-w_1(\lambda+\rho)}. \]

By taking cohomology of the map $\phi$, we get
\[H^5(\fg, K_\infty, J_\lambda\otimes \m_\lambda)\xrightarrow{\phi} H^5(\rS_\Gamma, \tm_\lambda\otimes \C).\]
Moreover, the map $\phi$ fits into the following diagram.
\[\begin{tikzcd}
H^5(\fg, K_\infty, J_\lambda\otimes \m_\lambda)
\arrow[bend left]{drr}{\Phi}
\arrow[bend right]{ddr}{\phi}
\arrow{dr}{\phi'} & & \\
&  H^5(\partial\rS_\Gamma, \tm_\lambda\otimes \C) \arrow{r}{\phi''}
& H^0(\rS_\Gamma^{\rM_1}, \tm_{w_{11}\cdot\lambda}\otimes \C) \\
& H^5(\rS_\Gamma, \tm_\lambda\otimes \C)\arrow{u}{r}
\end{tikzcd}\]
Here the map $\phi''$ is given as in Section \ref{case3}. In view of the constant term expansion, the map $\Phi$ is easily seen to be an isomorphism.  This completes the proof of part (2).

For the proof of part (1) and part (3), the same strategy applies and the proofs are indeed easier. In part  (1), the constant term, hence the Eisenstein series, has a double pole at $z_1=0, z_2=0$. By taking successive residues, the Langlands quotient we get is the constant representation, which provides nontrivial Eisenstein cohomology classes $H_{B, res}^{0}(\partial \rS_\Gamma, \tm_\lambda)$. While in case 5, the corresponding Eisenstein series is holomorphic at the special point $\theta_\lambda$. Hence part (3) can be proved by just taking cohomology of the map $\Eis$, see~\cite{Sch94} for more general cases.
 \end{proof}

\subsection{Proof of Theorem~\ref{Eiscoh}}  Now to finish the proof, we only need to interpret Proposition \ref{l-last2} using more concrete terms and then combine it with Proposition \ref{l-last1}.

According to Proposition \ref{l-last2}, when $m_1, m_2>0$, we always have
\begin{equation*}H^\bullet_{!, Eis}(\partial \rS_\Gamma, \tm_\lambda)=\bigoplus_{q=4}^7 H^q_{!}(\partial \rS_\Gamma, \tm_\lambda). \end{equation*}
Indeed, this also follows directly from the the main results of~\cite{LiSch04} by noticing that $\rank G_2-\rank K_\infty=0$.

\subsubsection{} When $m_1>0, m_2=0$, special attention needs to be paid on $H^q(\partial \rS_\Gamma, \tm_\lambda)(\pi^\psi, w)$ in case $w=w_7$, for which we have $\ell(w_7)=3$ and $\ell(w_7')=\ell(w_5)=2$. Thus, we have

\begin{equation*}
H^{q}_{Eis}(\rS_\Gamma, \tm_\lambda\otimes \C) =\left\{\begin{array}{cccc}
& H^q(\partial \rS_\Gamma, \tm_\lambda\otimes \C)  & q=5, 6, 7,\\
&\\
&0 & q=0, 1, 2, 3, \\
\end{array}\qquad\,, \right .
\end{equation*}
 and
\begin{align*}
  H^4_{!, Eis}(\rS_\Gamma, \tm_\lambda\otimes \C)(\pi^\psi, w)&= \C \psi,  \quad \mbox{if } L(1/2, \pi^\psi)=0, \\
  H^3_{!, Eis}(\rS_\Gamma, \tm_\lambda\otimes \C)(\pi^\psi, w)&= \C \psi, \quad \mbox{if } L(1/2, \pi^\psi)\ne 0.
\end{align*}

\subsubsection{} When $m_1=0, m_2>0$. A similar argument shows that
\begin{equation*}
H^{q}_{Eis}(\rS_\Gamma, \tm_\lambda\otimes \C) =\left\{\begin{array}{cccc}
& H^q(\partial \rS_\Gamma, \tm_\lambda\otimes \C)  & q=5, 6, 7,\\
&\\
&0 & q=0, 1, 2, 3, \\
\end{array}\qquad\,, \right .
\end{equation*}
 and
 \begin{align*}
  H^4_{!, Eis}(\rS_\Gamma, \tm_\lambda\otimes \C)(\pi^\psi, w)&= \C \psi,  \quad \mbox{if } L(1/2, Sym^3\pi^\psi)=0, \\
  H^3_{!, Eis}(\rS_\Gamma, \tm_\lambda\otimes \C)(\pi^\psi, w)&= \C \psi, \quad \mbox{if } L(1/2, Sym^3\pi^\psi)\ne 0.
\end{align*}

By combing above disscusion with Proposition \ref{l-last1}, Theorem \ref{Eiscoh} can be verified through a case-by-case study. \qed




\section*{Acknowledgements}
The authors would like to thank the Georg-August Universit\"at G\"ottingen,  Technische Universit\"at Dresden, and Max Planck Institue f\"ur Mathematics, Bonn, Germany, where much of the discussion and work on this project was accomplished, for its hospitality and wonderful working environment.
In addition, authors would like to extend their thanks to  G\"unter Harder for several discussions on the subject and the encouragement during the writing of this article. This work is financially supported by ERC Consolidator grants 648329 and 681207.




\nocite{}
\bibliographystyle{abbrv}
\bibliography{BG}

\end{document}